\documentclass[3p]{elsarticle}%
\usepackage{stmaryrd}
\usepackage{float}
\usepackage{amsmath,amsfonts,amssymb}
\usepackage{graphicx}
\usepackage{subfig}
\usepackage{color}
\usepackage{url}
\usepackage{hyperref}
\usepackage{algorithmic}
\usepackage{algorithm}
\usepackage{amsmath}
\usepackage{amsfonts}
\usepackage{amssymb}
\usepackage{amsthm}
\usepackage[latin1]{inputenc}
\usepackage{comment}
\usepackage[colorinlistoftodos]{todonotes}
\usepackage{bbm}
\usepackage{setspace}
\usepackage{enumerate}%
\providecommand{\U}[1]{\protect\rule{.1in}{.1in}}
\spacing{1.0}
\newtheorem{remark}{Remark}
\newtheorem{thm}{Theorem}
\newtheorem{prop}{Proposition}
\newtheorem{lem}{Lemma}

\newtheorem{defn}{Definition}

\allowdisplaybreaks
\begin{document}

\title{Analysis of obstacles immersed in viscous fluids using Brinkman's law for steady Stokes and Navier-Stokes equations}
\author[a,b]{Jorge Aguayo}
\ead{jaguayo@dim.uchile.cl}
\author[c,d]{Hugo Carrillo-Lincopi}
\ead{hugo.carrillo@inria.cl}
\address[a]{Mathematical Engineering Department, Faculty of Physical
and Mathematical Sciences, Universidad de Chile, Santiago, Chile}
\address[b]{Bernoulli Institute, University of Groningen, Groningen, The Netherlands}
\address[c]{Center for Mathematical Modelling, Universidad de Chile, Santiago, Chile}
\address[d]{Inria Chile Research Center, Av. Apoquindo 2827, Las Condes, Chile}
\date{\today}

\begin{abstract}
From the steady Stokes and Navier-Stokes models, a penalization method has been considered by several authors for approximating those fluid equations around obstacles. In this work, we present a justification for using fictitious domains to study obstacles immersed in incompressible viscous fluids through a simplified version of \textcolor{black}{Brinkman's} law for porous media. If the scalar function $\psi$ is considered as the inverse of permeability, it is possible to study the singularities of $\psi$ as approximations of obstacles (when $\psi$ tends to $\infty$) or of the domain corresponding to the fluid (when $\psi = 0$ or is very close to $0$). The strong convergence of the solution of the perturbed problem to the solution of the strong problem is studied, also considering error estimates that depend on the penalty parameter, both for fluids modeled with the Stokes and Navier-Stokes equations with \textcolor{black}{inhomogeneous} boundary conditions. A numerical experiment is presented that validates this result and allows to study the application of this perturbed problem simulation of flows and the identification of obstacles.
\end{abstract}

\maketitle
\pagestyle{myheadings}
\thispagestyle{plain}

\section{Introduction}
When modeling flows containing obstacles or enclosed by solid walls with a complex geometry, there are at least two main approaches: using body-fitted unstructured meshes to simulate the geometries or using a simplified mesh adding a penalization term in the differential equations. 

In numerical methods relying on the discretization with body-fitted geometry, solid walls are treated by Dirichlet boundary conditions on a mesh refined in the neighborhood of the wall. However, in this methods it is necessary to \textcolor{black}{rebuild} the meshes whenever the geometry changes, which could be a disadvantage for the \textcolor{black}{computing} performance. 


\textcolor{black}{The approach given by the addition of penalization terms has been reported in the pioneering work of Angot \cite{angot1990new} and \cite{khadra2000fictitious}, where the authors in addition show a numerical validation of the model}. Instead of considering Dirichlet boundary conditions on solid walls, in these methods the addition of a penalization or forcing term is considered in order to make the flow immovable inside the obstacles. The additional term can be seen as porosity Brinkman's law for imposing porous wall conditions \cite{B49} and it corresponds to the limit to null porosity. This method is versatile in terms of geometry: the mesh does not need to depend on the shape of the solid body, so that several geometries can be simulated in a simpler way. 

\textcolor{black}{Several extensions for Brinkman's penalty method have been studied, for example, the penalization was used to model the interface of multiphase flows \cite{BB15, bhalla2020simulating}, to study gas-particle flows coupling weakly compressible formulation of the Navier-Stokes equations with mass and heat transfer \cite{HDW19}, moving obstacles \cite{RAB07}, and penalizing Dirichlet or Neumann conditions applied on obstacle boundaries \cite{SYOS19,TNB21}. In  \cite{brown2014characteristic} the authors propose an extension to Brinkman penalization for generalized Neumann and Robin boundary conditions by introducing hyperbolic penalization terms with characteristics pointing inward on solid obstacles. In \cite{thirumalaisamy2021handling} the authors also study the Brinkman penalization method for Neumann and Robin boundary conditions. This method has been extended even for other equations, see for example \cite{kadoch2012volume} and \cite{ramiere2007fictitious}.
}


In addition, beyond the simulation of flows, the inverse problem of the obstacles or wall shapes estimation  \textcolor{black}{also can be studied considering the approaches mentioned above, that is, body-fitted unstructured meshes or the addition of a penalization term.} For the first approach, we can find, for example, works of \cite{B10} and \cite{conca2010detection}, where the authors provide identifiability and stability results, and \cite{alvarez2008identification} where the authors use shape derivatives arguments for the reconstruction. However, those methods need the geometry of the obstacle is not too complex, usually assuming a circular nature. For the second approach we can find works of \cite{aguayo2020} and \cite{fernandez2018noniterative} not depending on the geometry. However, either we study the direct or inverse problem, the penalized problem is seen as an approximation, so it is necessary to establish how accurate it is.

\textcolor{black}{In the literature, there are several works showing numerical validations} of the approximation between the penalized problem and the problem with the simulated geometry, for example \cite{jause2012numerical} and references above. \textcolor{black}{However, there are not much works showing in a theoretical way the effectiveness of the method as an approximation of obstacles. We mention} previous works in 
\cite{A99, ABF99}, where the authors \textcolor{black}{formally} established $\boldsymbol{H}^1$ error bounds for the approximate and exact problem in the steady Stokes system and unsteady Stokes, respectively. In both works, only Dirichlet conditions are considered for the entire domain boundary.

\textcolor{black}{In this article, we study the modeling of obstacles immersed in viscous fluids that satisfy the Stokes and Navier-Stokes equations for inhomogeneous boundary conditions, by the approximation of the fluid equations with the addition of a penalization term.  For the stationary Stokes problem the boundary conditions consist on: a known velocity entry, Dirichlet boundary conditions in the walls and a Neuman boundary condition in the outlet, while for the stationary Navier-Stokes we consider known velocities of entry and outlet, and Dirichlet boundary conditions for the walls. We
}
establish new convergence results, consisting in the steady Stokes and Navier-Stokes equations in the fluid domain, approximated by the respective penalized equations in a bigger domain containing both the fluid and solid part. \textcolor{black}{We follow techniques of \cite{A99, ABF99}, that is, we make problems to have homogeneous boundary conditions and we study the weak convergence of a sequence of functions depending on the penalization term in order to establish the strong convergence with rates depending on the penalization term going to infinite.}

\textcolor{black}{The results we present in this work have not been reported before in these particular settings, which have been chosen motivated by the applied problem of modeling blood flow in the presence of heart valves. In particular, this work justifies the use of a penalizing term in \cite{aguayo2020}, where the authors study the inverse problem of determining the geometry of heart valves given velocity measurements in the whole virtual domain at a given time, using as model the equations we present in this work. 
An important aspect to mention is that in such work the authors assume the velocity measurements consist of one snapshot obtained from magnetic resonance imaging (MRI) measurements using a technique known as phase-contrast MRI \cite{brown2014magnetic, kwong2008cardiovascular}, in which the time derivative of the velocity is assumed to be negligible due to the very short timescale of the data acquisition.} 

\textcolor{black}{The remaining of this article is organized as follows. In Section \ref{notations}, we provide the reader the basic notations of the fluid and solid domains, and the functional spaces involved in the main theorems. In Sections \ref{stokes} and \ref{ns},} we show estimates of the error in norm $\boldsymbol{H}^1$ induced by the penalization. 
In section \ref{stokes}, we show the analysis for the Stokes equations with mixed boundary conditions, which are the conditions usually considered in problems such that parts of the boundary are not walls. In Section \ref{ns}, we show the analysis for the Navier-Stokes case with \textcolor{black}{inhomogeneous} Dirichlet boundary conditions, which is also an improvement to \cite{A99}, to the nonlinear case. We closely follow ideas of \cite{A99, ABF99}. Finally, in Section \ref{num}, we show numerical tests to validate the theory in previous sections.

\section{Preliminaries and notations}\label{notations}

Consider a non-empty bounded domain $\Omega\subseteq\mathbb{R}^{d}$,
$d\in\left\{  2,3\right\}  $. The Lebesgue measure of $\Omega$ is denoted by
$\left\vert \Omega\right\vert $, which extends to lesser dimension spaces. The
norm and seminorms for Sobolev spaces $W^{m,p}\left(  \Omega\right)  $ is
denoted by $\left\Vert \cdot\right\Vert _{m,p,\Omega}$ and $\left\vert
\cdot\right\vert _{m,p,\Omega}$, respectively. For $p=2$, the norm, seminorms
and inner product of the space $W^{m,2}\left(  \Omega\right)  =H^{m}\left(
\Omega\right)  $ are denoted by $\left\Vert \cdot\right\Vert _{m,\Omega}$,
$\left\vert \cdot\right\vert _{m,\Omega}$ and $\left(  \cdot,\cdot\right)
_{m,\Omega}$, respectively. Also, $\left\Vert \cdot\right\Vert _{\infty
,\Omega}$ denotes the norm of $L^{\infty}\left(  \Omega\right)  $. The spaces
$\boldsymbol{H}^{m}\left(  \Omega\right)  $ and $\boldsymbol{W}^{m,p}\left(
\Omega\right)  $ are defined by $\boldsymbol{H}^{m}\left(  \Omega\right)
=\left[  H^{m}\left(  \Omega\right)  \right]  ^{d}$ and $\boldsymbol{W}%
^{m,p}\left(  \Omega\right)  =\left[  W^{m,p}\left(  \Omega\right)  \right]
^{d}$. The notation for norms, seminorms and inner products will be extended
from $W^{m,p}\left(  \Omega\right)  $ or $H^{m}\left(  \Omega\right)  $.

We assume that $\Gamma=\partial\Omega$ is piecewise $\mathcal{C}^{1}$ and $\Omega$ contains $N$ regular obstacles given by nonempty \textcolor{black}{open} sets $\Omega_S^{j}\subseteq\Omega$ for $j\in\left\{  1,\ldots,N\right\}  $.

\begin{defn}
The sets $\Omega_S$, $\Omega_F$, $\Sigma_{S}^{j}$ (for $j\in\left\{
1,\ldots,N\right\}  $) are defined by
\[
\Omega_S=\bigcup_{j=1}^{N}\Omega_S^{j},\quad\Omega_F=\Omega
\setminus\overline{\Omega}_{s}%
\]%
\[
\Sigma_{s}^{i}=\partial\Omega_S^{j},\quad\Gamma=\partial\Omega
\]
We also define $\Gamma_{I},\Gamma_{W},\Gamma_{O}\subset\Gamma$, disjoint
subsets of $\Gamma$, such that
\[
\overline{\Gamma}_{I}\cup\overline{\Gamma}_{W}\cup\overline{\Gamma}_{O}%
=\Gamma,
\]%
\[
\Gamma_{I}\subset\partial\Omega_F\cap\partial\Omega,\quad\Gamma_{O}%
\subset\partial\Omega_F\cap\partial\Omega
\]
Finally, we define $\Gamma_{F,W}=\partial\Omega_F\setminus(\Gamma_{I}%
\cup\Gamma_{O})$
\end{defn}

The set $\Omega_F$ models a fluid domain where the Stokes or Navier-Stokes
equations is fulfilled.

\begin{figure}[H]
\centering\includegraphics[height=4.5cm,keepaspectratio]{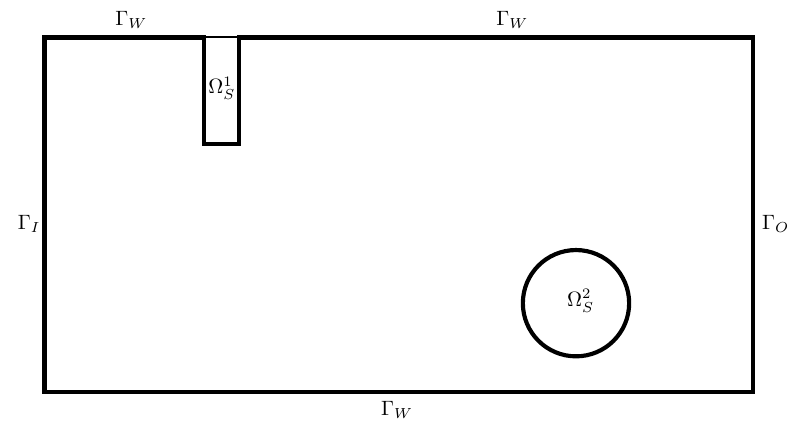}\caption{Example
of $\Omega$ with two obstacles $\Omega_S^{1}$ and $\Omega_S^{2}$.}%
\label{fig1}%
\end{figure}

\begin{defn}
Let $\gamma\subseteq\partial\Omega$. We define the following spaces%
\begin{align*}
\boldsymbol{H}_{0}^{1}(\Omega)  &  =\left\{  \boldsymbol{v}%
\in\boldsymbol{H}^{1}\left(  \Omega\right)  \mid\boldsymbol{v}=\boldsymbol{0}%
\text{ on }\partial\Omega\right\}\\
\boldsymbol{H}_{\gamma}^{1}(\Omega)  &  =\left\{  \boldsymbol{v}%
\in\boldsymbol{H}^{1}\left(  \Omega\right)  \mid\boldsymbol{v}=\boldsymbol{0}%
\text{ on }\partial\Omega\setminus\gamma\right\} \\
\boldsymbol{H}_{\operatorname{div}}(\Omega)  &  =\left\{  \boldsymbol{v}%
\in\boldsymbol{H}^{1}\left(  \Omega\right)  \mid\operatorname{div}%
\boldsymbol{v}=\boldsymbol{0}\text{ on }\Omega\right\} \\
\boldsymbol{V}_{\gamma}(\Omega)  &  =\boldsymbol{H}_{\gamma}^{1}(\Omega
)\cap\boldsymbol{H}_{\operatorname{div}}(\Omega)\\
\boldsymbol{V} (\Omega)  &  =\boldsymbol{H}_{0}^{1}(\Omega
)\cap\boldsymbol{H}_{\operatorname{div}}(\Omega)\\
L_{0}^{2}\left(  \Omega\right)   &  =\left\{  p\in L^{2}\left(  \Omega\right)
\mid\left(  p,1\right)  _{0,\Omega}=0\right\} \\
\mathcal{H}_{\gamma}\left(  \Omega\right)   &  =\boldsymbol{V}_{\gamma}%
(\Omega)\times L_{0}^{2}\left(  \Omega\right) \\
\mathcal{H}\left(  \Omega\right)   &  =\boldsymbol{V} (\Omega)\times
L_{0}^{2}\left(  \Omega\right)
\end{align*}

We extend these definitions to $\Omega_F$ and $\Omega_S$.
\end{defn}

\section{Stokes system with mixed boundary conditions}\label{stokes}

Let $\nu>0$, $\boldsymbol{u}_{D}\in\boldsymbol{H}^{1/2}(\Gamma_{I})$ such that
$\boldsymbol{u}_{D}=\boldsymbol{0}$ on $\overline{\Gamma}_{I} \cap
\overline{\Gamma}_{W}$ and $(\boldsymbol{u}, p) \in\mathcal{H}\left(
\Omega\right)  $ the unique solution of the Stokes system with mixed
boundary conditions over $\Omega_F$ given by

\begin{alignat}{2}
                           -\nu\triangle\boldsymbol{u} + \nabla p & = \boldsymbol{0}     & \quad &\text{in }\Omega_F \label{StokesSystem}\\
                                \operatorname{div} \boldsymbol{u} & =  0             		 & 			 &\text{in } \Omega_F\nonumber\\
                                                   \boldsymbol{u} & = \boldsymbol{u}_{D} &       &\text{on } \Gamma_{I}\nonumber\\
                                                   \boldsymbol{u} & = \boldsymbol{0}     &       &\text{on } \Gamma_{F,W}\nonumber\\
-\nu\dfrac{\partial\boldsymbol{u}}{\partial n} + p \boldsymbol{n} & = \boldsymbol{0}     &       &\text{on } \Gamma_{O}\nonumber
\end{alignat}
extended by $(\boldsymbol{0}, 0)$ in $\Omega_S$.

On the other hand, for each $R>0$, let $(\boldsymbol{u}_{R}, p_{R})
\in\mathcal{H}\left(  \Omega\right)  $ the unique solution of the modified Stokes system with a $L^{2}$ penalization term over $\Omega$ given by

\begin{alignat}{2} 
-\nu\triangle\boldsymbol{u}_{R} + \nabla p_{R} + \psi_{R} \boldsymbol{u}_{R} & = \boldsymbol{0}     & \quad &\text{in }\Omega \label{Penalized_System}\\
                                       \operatorname{div} \boldsymbol{u}_{R} & =  0                 &       &\text{in } \Omega\nonumber\\
                                                            \boldsymbol{u}_R & = \boldsymbol{u}_{D} &       &\text{on } \Gamma_{I}\nonumber\\
                                                            \boldsymbol{u}_R & = \boldsymbol{0}     &       &\text{on } \Gamma_{W}\nonumber\\
       -\nu\dfrac{\partial\boldsymbol{u}_R}{\partial n} + p_R \boldsymbol{n} & = \boldsymbol{0}     &       &\text{on } \Gamma_{O}\nonumber
\end{alignat}

where $\psi_{R} = R \chi_{\Omega_S}$, and
\[
\chi_{\Omega_S} (x) =
\begin{cases}
1 & \text{if} \ x \in\Omega_S ,\\
0 & \text{otherwise.}%
\end{cases}
\]


\subsection{Previous results}

We start giving some previous results in order to use them in the proof of Theorems \ref{stokes1} and \ref{stokes2}.

\begin{lem}
\label{relevo1} Let $\eta>0$ and $M>0$ such that $\left\Vert \boldsymbol{u}%
_{D}\right\Vert _{1/2,\Gamma_{I}}\leq M$. There exists $\boldsymbol{g}%
\in\boldsymbol{H}^{1}\left(  \Omega_F\right)  $ such that
$\operatorname{div}\boldsymbol{g}=0$, $\boldsymbol{g}=\boldsymbol{u}_{D}$ on
$\Gamma_{I}$, $\boldsymbol{g}=\boldsymbol{0}$ on $\Gamma_{F,W}$ and
\[
\left(  \forall\boldsymbol{v}\in\boldsymbol{H}^{1}\left(  \Omega_F\right)
\right)  \text{\quad}\left\vert \left(  \left(  \nabla\boldsymbol{g}\right)
\boldsymbol{v},\boldsymbol{v}\right)_{0,\Omega_F}  \right\vert \leq\eta\left\vert
\boldsymbol{v}\right\vert _{1,\Omega_F}^{2}%
\]
and a constant $c>0$, that only depends of $\Omega$, $\Gamma_{I}$, $\Gamma
_{W}$ and $M$ such that
\[
\left\Vert \boldsymbol{g}\right\Vert _{1,\Omega}\leq c\left\Vert
\boldsymbol{u}_{D}\right\Vert _{1/2,\Gamma_{I}}%
\]

\end{lem}

\begin{proof}
Let $\boldsymbol{u}_{D}^{\ast}\in\boldsymbol{H}^{1/2}\left(  \partial
\Omega_F\right)  $, with $\boldsymbol{u}_{D}^{\ast}$ an extension of
$\boldsymbol{u}_{D}$ such that $\left\Vert \boldsymbol{u}_{D}^{\ast
}\right\Vert _{1/2,\partial\Omega_F}\leq2\left\Vert \boldsymbol{u}%
_{D}\right\Vert _{1/2,\Gamma_{I}}$ and
\[
\int_{\partial\Omega_F}\boldsymbol{u}_{D}^{\ast}\cdot\boldsymbol{n}\text{
}dS=0
\]
Applying Lemma IV.2.3 in \cite{GR86}, there exists $\boldsymbol{g}%
\in\boldsymbol{H}^{1}\left(  \Omega_F\right)  $ such that
$\operatorname{div}\boldsymbol{g}=0$ in $\Omega_F$, $\boldsymbol{g}%
=\boldsymbol{u}_{D}^{\ast}$ on $\partial\Omega_F$ and%
\[
\left(  \forall\boldsymbol{v}\in\boldsymbol{H}^{1}\left(  \Omega_F\right)
\right)  \text{\quad}\left\vert \left(  \left(  \nabla\boldsymbol{g}\right)
\boldsymbol{v},\boldsymbol{v}\right)  _{0,\Omega_F}\right\vert \leq
\eta\left\vert \boldsymbol{v}\right\vert _{1,\Omega_F}^{2}%
\]
In particular, $\boldsymbol{g}=\boldsymbol{u}_{D}$ on $\Gamma_{I}$ and
$\boldsymbol{g}=\boldsymbol{0}$ on $\Gamma_{W}$, proving the first part of the
lemma. Using Lemma IX.4.2 in \cite{G11}, we can deduce the existence of $c$.
\end{proof}


\begin{remark}
Since $\boldsymbol{g}=\boldsymbol{0}$ on $\Gamma_{F,W}$, it is possible to
extend $\boldsymbol{g}\in\boldsymbol{H}^{1}\left(  \Omega_F\right)  $ to
$\boldsymbol{g}\in\boldsymbol{H}^{1}\left(  \Omega\right)  $ such that
$\boldsymbol{g}=\boldsymbol{u}_{D}$ on $\Gamma_{I}$ and $\boldsymbol{g}%
=\boldsymbol{0}$ on $\Gamma_{W}$.
\end{remark}

\begin{prop}
\label{prop1} Let $\boldsymbol{v}_{R} = \boldsymbol{u}_{R}-\boldsymbol{g}$,
where $\boldsymbol{g}$ is given by Lemma \ref{relevo1}. Then
\[
\vert\boldsymbol{v}_{R}\vert_{1, \Omega} \leq\vert\boldsymbol{g}\vert_{1, \Omega}%
\quad\text{and}\quad\Vert\boldsymbol{v}_{R} \Vert_{0, \Omega_S} \leq\dfrac{\nu}{R}
\vert\boldsymbol{g}\vert_{1,\Omega}
\]

\end{prop}

\begin{proof}
Consider the penalized equation \eqref{Penalized_System} after introducing $\boldsymbol{v}_R$ given by
\begin{alignat*}{2} 
- \nu \triangle \boldsymbol{v}_R + \nabla p_R + \psi_R \boldsymbol{v}_R  & = \nu \triangle \boldsymbol{g}&\quad & \text{in } \Omega \\
                                          \text{div} (\boldsymbol{v}_R)  & = 0                           & & \text{in } \Omega \\
                                                       \boldsymbol{v}_R  & =  0                          & & \text{on } \Gamma_{I} \cup \Gamma_{W} \\
-\nu \dfrac{\partial \boldsymbol{v}_R}{\partial n} + p_R \boldsymbol{n}  & =  \nu \dfrac{\partial \boldsymbol{g}}{\partial \boldsymbol{n}} & & \text{on } \Gamma_{O} \\
\end{alignat*}
Testing this equations by $\boldsymbol{w} \in \boldsymbol{V}_{\Gamma_O} (\Omega)$ and $q \in L_0^2 (\Omega)$, we obtain
\begin{equation} \label{VF1}
\nu (\nabla \boldsymbol{v}_R, \nabla \boldsymbol{w})_{0, \Omega} + R (\boldsymbol{v}_R , \boldsymbol{w})_{0, \Omega_S} 
= - \nu (\nabla \boldsymbol{g}, \nabla \boldsymbol{w})_{0, \Omega}
\end{equation}
Taking  $\boldsymbol{w} = \boldsymbol{v}_R $, we deduce
$$
\nu \vert\boldsymbol{v}_R\vert_{1, \Omega}^2 + R \Vert\boldsymbol{v}_R\Vert_{0, \Omega_S}^2  = - \nu (\nabla \boldsymbol{g}, \nabla \boldsymbol{v}_R)_{0, \Omega} \leq \nu \vert\boldsymbol{g}\vert_{1, \Omega} \vert\boldsymbol{v}_R\vert_{1, \Omega}
$$
and then we conclude.
\end{proof}

\begin{prop}
$\boldsymbol{v}_{R} $ converges weakly to $\boldsymbol{v}$ in $\boldsymbol{V}%
_{\Gamma_{O}} (\Omega)$.
\end{prop}

\begin{proof}
By the result of Proposition \ref{prop1}, we see that there exists a subsequence $\boldsymbol{v}_R$ (we call it the same way) weakly convergent in $\boldsymbol{H}^1 (\Omega)$ to $\tilde{\boldsymbol{v}}$. Since $\tilde{\boldsymbol{v}} =0$ in $\Omega_S$, applying Trace Theorem (see Theorem II.4.1 in \cite{G11}), we can see that $\tilde{\boldsymbol{v}} = 0 $ on $\partial \Omega_S$.
Later, for all $\boldsymbol{w} \in \boldsymbol{V}_{\Gamma_O} (\Omega)$ we have
$$
(\psi_R \boldsymbol{v}_R , \boldsymbol{w})_{0, \Omega}  = - \Big( \nu (\nabla \boldsymbol{g}, \nabla \boldsymbol{w})_{0, \Omega} + \nu (\nabla \boldsymbol{v}_R, \nabla \boldsymbol{w})_{0, \Omega} \Big) \rightarrow  - \Big( \nu (\nabla \boldsymbol{g}, \nabla \boldsymbol{w})_{0, \Omega} + \nu (\nabla \tilde{\boldsymbol{v}}, \nabla \boldsymbol{w})_{0, \Omega} \Big)
$$
Hence $\psi_R \boldsymbol{v}_R$ converges weakly to some $\boldsymbol{h} \in [\boldsymbol{V}_{\Gamma_O} (\Omega)]'$, where $\text{supp} (\boldsymbol{h}) \subseteq \Omega$. Then,
\begin{equation} \label{FVtildev}
\nu (\nabla \tilde{\boldsymbol{v}}, \nabla \boldsymbol{w} )_{0, \Omega} + \langle \boldsymbol{h}, \boldsymbol{w} \rangle_{\boldsymbol{H}^{-1}(\Omega),\boldsymbol{H}^{1}(\Omega)} = - \nu (\nabla \boldsymbol{g}, \nabla \boldsymbol{w})_{0, \Omega}
\end{equation}
Since $\boldsymbol{v}_R = 0$ on $\Gamma_{I} \cup \Gamma_{W}$, we have $\tilde{\boldsymbol{v}} = 0$ on $\Gamma_{I} \cup \Gamma_{W}$ as well, by the continuity of the trace operator.
Now, applying the De Rham's Theorem (see Theorem I.2.3 in \cite{GR86}), there exists $\tilde{p} \in L_{0}^2 (\Omega)$ such that
\begin{alignat*}{2}
-\nu \triangle \tilde{\boldsymbol{v}} + \nabla \tilde{p} + \boldsymbol{h} & = \nu \triangle \boldsymbol{g} &\quad& \text{in } \Omega \\
\text{div} \tilde{\boldsymbol{v}} & =  0 & &\text{in } \Omega \\
\tilde{\boldsymbol{v}} & = \boldsymbol{0} & & \text{on } (\Gamma_{I} \cup \Gamma_{W}) \cap \partial \Omega_F
\end{alignat*}
Taking $(\boldsymbol{w}, q) \in \mathcal{H}(\Omega)$ such that $(\boldsymbol{w}, q) = (\boldsymbol{0}, 0)$ in $\Omega_S$, we have
\begin{equation} \label{VF2}
\nu (\nabla \tilde{\boldsymbol{v}}, \nabla \boldsymbol{w})_{0, \Omega_F} + \Big( - \nu \dfrac{\partial (\tilde{\boldsymbol{v}} + \boldsymbol{g})}{\partial \boldsymbol{n}} + p \boldsymbol{n} , \boldsymbol{w} \Big)_{0, \Gamma_{O}} = - \nu (\nabla \boldsymbol{g}, \nabla \boldsymbol{w})_{0, \Omega_F}
\end{equation}
and then
$$
-\dfrac{\partial (\tilde{\boldsymbol{v}} + \boldsymbol{g})}{\partial \boldsymbol{n}} + p \boldsymbol{n} = \boldsymbol{0}
$$
on $\Gamma_{O}$. Hence, $(\tilde{\boldsymbol{v}}, \tilde{p})$ is a weak solution for Equation \eqref{StokesSystem}. Since Equation \eqref{StokesSystem} has a unique solution, we conclude $(\tilde{\boldsymbol{v}}, \tilde{p}) = (\boldsymbol{v}, p)$. Finally, extending the solution by $(\boldsymbol{0},0)$ in $\Omega_S$, we have that for all $(\boldsymbol{w}, q) \in \mathcal{H} (\Omega)$, that is,
$$
\nu (\nabla \tilde{\boldsymbol{v}} , \nabla \boldsymbol{w})_{0, \Omega}  = - \nu (\nabla \boldsymbol{g} , \nabla \boldsymbol{w})_{0, \Omega}
$$
In conclusion, $(\tilde{\boldsymbol{v}}, \tilde{p}) = (\boldsymbol{v}, p)$ in $\Omega$ and $\tilde{\boldsymbol{v}}_R \rightharpoonup \boldsymbol{v}$ in $\boldsymbol{H}_{\Gamma_{O}}^1 (\Omega)$ as $R \rightarrow \infty$.
\end{proof}

\subsection{Main results}
Now we can establish the first convergence result.

\begin{thm}\label{stokes1}
Let $R>0$, $\boldsymbol{u}$ be solution of \eqref{StokesSystem} and $\boldsymbol{u}_{R}$ solution of
\eqref{Penalized_System}. With the previous assumptions, there is strong convergence of $\{\boldsymbol{u}_{R}%
\}_{R>0}$, that is
\[
\lim_{R\rightarrow\infty}\vert\boldsymbol{u}_{R}-\boldsymbol{u}\vert_{1,\Omega}=0
\]
and there there exists a constant $C > 0$ independent such that for all $R>0$
\[
\Vert\boldsymbol{u}-\boldsymbol{u}_{R}\Vert_{0,\Omega_S}\leq\dfrac{C}{R^{1/2}}%
\]

\end{thm}

\begin{proof}
Let $\boldsymbol{w}_R = \boldsymbol{v}_R - \boldsymbol{v}$. Subtracting the variational formulations \eqref{VF1} and \eqref{FVtildev} for $\boldsymbol{v}_R$ and $\tilde{\boldsymbol{v}}$, respectively, we obtain for all $\boldsymbol{w} \in \boldsymbol{V}$:
\begin{equation} \label{FV_wR}
\nu (\nabla \boldsymbol{w}_R, \nabla \boldsymbol{w})_{0, \Omega} + R (\boldsymbol{w}_R , \boldsymbol{w})_{0, \Omega_S} =  \langle \boldsymbol{h}, \boldsymbol{w}\rangle_{\boldsymbol{V}', \boldsymbol{V}}
\end{equation}
since $\boldsymbol{v} = \boldsymbol{g} = \boldsymbol{0}$ in $\Omega_S$. Taking $\boldsymbol{w} = \boldsymbol{w}_R$ and using that $\boldsymbol{w}_R \rightharpoonup 0$ in $\boldsymbol{V}$ as $R \rightarrow \infty$, we obtain
$$
\nu \vert\boldsymbol{w}_R\vert_{1, \Omega}^2 + R \Vert \boldsymbol{w}_R\Vert_{0, \Omega_S}^2 = \langle \boldsymbol{h}, \boldsymbol{w}_R \rangle_{\boldsymbol{V}', \boldsymbol{V}} \rightarrow 0
$$
proving the theorem.
\end{proof}

Imposing more regularity to $\boldsymbol{u}_{D}$ and $\partial\Omega_F$, the first convergence theorem can be upgraded to this new result.

\begin{thm}\label{stokes2}
Let $R>0$, $\boldsymbol{u}$ be solution of \eqref{StokesSystem} and $\boldsymbol{u}_{R}$ solution of \eqref{Penalized_System}. With the previous assumptions, where we assume in addition that $\partial\Omega_F$ is piecewise $\mathcal{C}^{2}$ class and
$\boldsymbol{u}_{D} \in H^{3/2} (\Gamma_I)$, then there is strong convergence of $\{\boldsymbol{u}_{R}\}_{R>0}$ in $\boldsymbol{H}^{1}(\Omega)$ when $R\rightarrow\infty$. Furthermore, there exists a constant $C > 0$ independent of $R$ such that for all $R>0$
\[
\vert \boldsymbol{u} - \boldsymbol{u}_{R} \vert_{1, \Omega} \leq\dfrac{C}{R^{1/4}},
\qquad\Vert \boldsymbol{u} - \boldsymbol{u}_{R} \Vert_{0, \Omega_S} \leq\dfrac
{C}{R^{3/4}} .
\]

\end{thm}

\begin{proof}
We can assume that the function $\boldsymbol{g}$ given by Lemma \ref{relevo1} is now in $\boldsymbol{H}^{2}\left(
\Omega\right)  $. Hence, results about regularity of solution to the Stokes equations (see Theorem IV.6.1 in \cite{G11}) allow us to consider $\left(  \boldsymbol{v},p\right)  \in
\boldsymbol{H}^{2}\left(  \Omega\right)  \times\boldsymbol{H}^{1}\left(
\Omega\right)  $. Let us replace the Dirichlet condition of $\boldsymbol{u}$ in (6) on $\partial \Omega_S \setminus \Gamma$ by
\[
-\nu\dfrac{\partial\boldsymbol{u}}{\partial\boldsymbol{n}}+p\boldsymbol{n} = \boldsymbol{k}.
\]
Then, for $\boldsymbol{v} = \boldsymbol{u} - \boldsymbol{g}$,
\[
\boldsymbol{k}=-\nu\dfrac{\partial\boldsymbol{v}+\boldsymbol{g}}%
{\partial\boldsymbol{n}}+p\boldsymbol{n}\text{\quad on }\partial\Omega_S \setminus \Gamma .
\]
where $\boldsymbol{k}\in\boldsymbol{H}^{1/2}(\Gamma_O)$. For all $\boldsymbol{w} \in \boldsymbol{H}_{(\partial \Omega_S \setminus \Gamma) \cup \Gamma_O}^1 (\Omega_F)$:
\[
\nu\left(  \nabla\boldsymbol{v}%
,\nabla\boldsymbol{w}\right)  _{0,\Omega_F}+\left(  \boldsymbol{k}%
,\boldsymbol{w}\right)  _{0,\partial\Omega_S}=-\nu\left(  \nabla
\boldsymbol{g},\nabla\boldsymbol{w}\right)  _{0,\Omega_F}%
\]
Since
\[
\left(  \forall\boldsymbol{w}\in\boldsymbol{H}_{\Gamma_O}^{1}\left(
\Omega\right)  \right)  \text{\qquad}\nu\left(  \nabla\boldsymbol{v}%
,\nabla\boldsymbol{w}\right)  _{0,\Omega}+\left\langle \boldsymbol{h}%
,\boldsymbol{w}\right\rangle _{V^{\prime},V}=-\nu\left(  \nabla\boldsymbol{g}%
,\nabla\boldsymbol{w}\right)  _{0,\Omega}%
\]
and considering $\boldsymbol{v}=\boldsymbol{g}=\boldsymbol{0}$ in $\Omega_S$, then we have
\begin{equation} \label{h_and_k}
\left(  \forall\boldsymbol{w}\in\boldsymbol{H}_{\Gamma_O}^{1}\left(
\Omega\right)  \right)  \text{\qquad}\left\langle \boldsymbol{h}%
,\boldsymbol{w}\right\rangle _{V^{\prime},V}=\left(  \boldsymbol{k}%
,\boldsymbol{w}\right)  _{0,\partial\Omega_S \setminus \Gamma}%
\end{equation}
Then, from \eqref{FV_wR} and \eqref{h_and_k}, applying H\"{o}lder and Cauchy-Schwartz inequalities, we have
\begin{align*}
\nu\left\vert \boldsymbol{w}_{R}\right\vert _{1,\Omega}^{2}+R\left\Vert
\boldsymbol{w}_{R}\right\Vert _{0,\Omega_S}^{2}  & =\left(  \boldsymbol{k}%
,\boldsymbol{w}_{R}\right)  _{0,\partial\Omega_S}\\
& \leq C\left\Vert \boldsymbol{k}\right\Vert _{0,\partial\Omega_S}\left\vert
\boldsymbol{w}_{R}\right\vert _{1,\Omega_S}^{1/2}\left\Vert \boldsymbol{w}%
_{R}\right\Vert _{0,\Omega_S}^{1/2}\\
& \leq\dfrac{\left(  C\left\Vert \boldsymbol{k}\right\Vert _{0,\partial
\Omega_S}\right)  ^{2}}{2\left(  \nu R\right)  ^{1/2}}+\dfrac{1}{4}\left(
\nu\left\vert \boldsymbol{w}_{R}\right\vert _{1,\Omega}^{2}+R\left\Vert
\boldsymbol{w}_{R}\right\Vert _{0,\Omega_S}^{2}\right)
\end{align*}
where $C>0$ is a constant independent of $R$. Then,
\[
\nu\left\vert \boldsymbol{w}_{R}\right\vert _{1,\Omega}^{2}+R\left\Vert
\boldsymbol{w}_{R}\right\Vert _{0,\Omega_S}^{2}\leq\dfrac{2\left(
C\left\Vert \boldsymbol{k}\right\Vert _{0,\partial\Omega_S}\right)  ^{2}%
}{3\left(  \nu R\right)  ^{1/2}}%
\]
In conclusion, $\left\vert \boldsymbol{w}_{R}\right\vert _{1,\Omega
}=\mathcal{O}\left(  R^{-1/4}\right)  $ and $\left\Vert \boldsymbol{w}%
_{R}\right\Vert _{0,\Omega_S}=\mathcal{O}\left(  R^{-3/4}\right)  $, proving this theorem.
\end{proof}

\section{Navier-Stokes equation with Dirichlet boundary conditions}\label{ns}

\subsection{Previous results}

Let us consider $\Omega, \Omega_F$ and $\Omega_S$ as described in Section~\ref{stokes}. Let $(\boldsymbol{u}, p) \in\mathcal{H} (\Omega_F)$ a solution of the Navier-Stokes system with Dirichlet boundary conditions over $\Omega_F$

\begin{alignat}{2}  %
-\nu\triangle\boldsymbol{u}+\left(  \nabla\boldsymbol{u}\right)\boldsymbol{u}+\nabla p & = \boldsymbol{0} &\quad& \text{in }\Omega_F\label{NS}\\
\operatorname{div}\boldsymbol{u} & = 0 & &\text{in }\Omega_F\nonumber\\
\boldsymbol{u} & = \boldsymbol{u}_{D} & & \text{on }\Gamma_{I}\cup\Gamma
_{O}\nonumber\\
\boldsymbol{u} & = \boldsymbol{0} & &\text{on } \partial\Omega_F\setminus\left(  \Gamma_{I} \cup\Gamma_{O}\right)\nonumber
\end{alignat}
and for all $R>0$, let $(\boldsymbol{u}_{R}, p_{R}) \in\mathcal{H} (\Omega)$
a solution of the following modified Navier-Stokes system over $\Omega$, with a $L^{2}$
penalization term, given by

\begin{alignat}{2}
-\nu\triangle\boldsymbol{u}_{R}+\left(  \nabla\boldsymbol{u}_{R}\right)
\boldsymbol{u}_{R} +\nabla p_{R} +R\chi_{\Omega_S}\boldsymbol{u}_{R}  & = \boldsymbol{0} &\quad& \text{in }\Omega\label{NS_R}\\
\operatorname{div}\boldsymbol{u}_R & = 0 & &\text{in }\Omega_F\nonumber\\
\boldsymbol{u}_R & = \boldsymbol{u}_{D} & & \text{on }\Gamma_{I}\cup\Gamma_{O}\nonumber\\
\boldsymbol{u}_R & = \boldsymbol{0} & &\text{on } \Gamma_{W}\nonumber
\end{alignat}
provided $\boldsymbol{u}_{D}\in\boldsymbol{H}^{1/2}\left(  \Omega\right)  $
and
\[
\int_{\Gamma_{I}}\boldsymbol{u}_{D}\cdot\boldsymbol{n}\text{ }dS+\int%
_{\Gamma_{O}}\boldsymbol{u}_{D}\cdot\boldsymbol{n}\text{ }dS=0
\]

The uniqueness of solution of both problems is guaranteed under certain additional hypotheses. In order to establish those hypotheses, it is necessary to cite the following results.

\begin{thm}
\label{Ineq1_NS} There exists $\boldsymbol{g}\in H^{1}\left(  \Omega
_{f}\right)  $ such that $\operatorname{div}\boldsymbol{g}=0$, $\boldsymbol{g}%
=\boldsymbol{u}_{D}$ on $\Gamma_{I}\cup\Gamma_{0}$, $\boldsymbol{g}%
=\boldsymbol{0}$ on $\Gamma_{f} \setminus\left(  \Gamma_{I} \cup\Gamma
_{O}\right)  $ and
\[
\left(  \forall\boldsymbol{w}\in\boldsymbol{H}_{0}^{1}\left(  \Omega\right)
\right)  \qquad\left\vert \left(  \left(  \nabla\boldsymbol{g}\right)
\boldsymbol{w},\boldsymbol{w}\right)  \right\vert \leq\alpha\left\vert
\boldsymbol{w}\right\vert _{1,\Omega}^{2}%
\]
for all $\alpha\in\left(  0,\nu\right)  $.
\end{thm}
\begin{proof}
See Lemma IV.2.3 in \cite{GR86} and Lemma IX.4.2 in \cite{G11}.
\end{proof}

\begin{thm}
\label{Ineq2_NS} There exists a constant $\kappa>0$ only depending on $\Omega
$, such that
\begin{align*}
\left(  \forall\left(  \boldsymbol{u},\boldsymbol{v},\boldsymbol{w}\right)
\in\boldsymbol{H}_{0}^{1}\left(  \Omega\right)  \times\boldsymbol{H}%
^{1}\left(  \Omega\right)  \times\boldsymbol{H}_{0}^{1}\left(  \Omega\right)
\right)  \qquad\left\vert \left(  \left(  \nabla\boldsymbol{v}\right)
\boldsymbol{u},\boldsymbol{w}\right)  \right\vert  &  \leq\kappa\left\vert
\boldsymbol{u}\right\vert _{1,\Omega}\left\vert \boldsymbol{v}\right\vert
_{1,\Omega}\left\vert \boldsymbol{w}\right\vert _{1,\Omega}\\
\left(  \forall\left(  \boldsymbol{u},\boldsymbol{v},\boldsymbol{w}\right)
\in\boldsymbol{H}^{1}\left(  \Omega\right)  \times\boldsymbol{H}^{1}\left(
\Omega\right)  \times\boldsymbol{H}_{0}^{1}\left(  \Omega\right)  \right)
\qquad\left\vert \left(  \left(  \nabla\boldsymbol{v}\right)  \boldsymbol{u}%
,\boldsymbol{w}\right)  \right\vert  &  \leq\kappa\left\Vert \boldsymbol{u}%
\right\Vert _{1,\Omega}\left\vert \boldsymbol{v}\right\vert _{1,\Omega
}\left\vert \boldsymbol{w}\right\vert _{1,\Omega}%
\end{align*}

\end{thm}

\begin{proof}
Direct consequence of Hölder inequalty and Sobolev Embedding Theorem. 
\end{proof}

Then, as for the Stokes problem, we consider the extension of $\boldsymbol{g}$
to $\boldsymbol{H}^{1}\left(  \Omega\right)  $ such that $\boldsymbol{g}%
=\boldsymbol{0}$ en $\Omega_S$. Let us define $\boldsymbol{v}=\boldsymbol{u}%
-\boldsymbol{g}$, so we have the following equation
\begin{alignat}{2}
-\nu\triangle\boldsymbol{v}+\left(  \nabla\boldsymbol{v}\right)
\boldsymbol{v}+\left(  \nabla\boldsymbol{v}\right)  \boldsymbol{g}+\left(
\nabla\boldsymbol{g}\right)  \boldsymbol{v}+\nabla p & = \nu\triangle
\boldsymbol{g}-\left(  \nabla\boldsymbol{g}\right)  \boldsymbol{g} &\quad&\text{in }\Omega_F\label{NS_relevo}\\
\operatorname{div}\boldsymbol{v} & = 0 & & \text{in }\Omega_F\nonumber\\
\boldsymbol{v} & = \boldsymbol{0} & &\text{on }\partial\Omega_F\nonumber
\end{alignat}
where we extend $(\boldsymbol{v}, p) \in\mathcal{H}(\Omega) $ by
$(\boldsymbol{0}, 0)$. In addition, let $\boldsymbol{v}_{R}=\boldsymbol{u}%
_{R}-\boldsymbol{g}$, so we have:

\begin{alignat}{2} %
-\nu\triangle\boldsymbol{v}_{R}+\left(  \nabla\boldsymbol{v}_{R}\right)\boldsymbol{v}_{R}+\left(  \nabla\boldsymbol{v}_{R}\right)  \boldsymbol{g}
+\left(  \nabla\boldsymbol{g}\right)  \boldsymbol{v}_{R}+\nabla p_{R}+R\chi_{\Omega}\boldsymbol{v}_{R} & =  \nu\triangle\boldsymbol{g}-\left(\nabla\boldsymbol{g}\right)  \boldsymbol{g} &\quad & \text{in }\Omega\label{NS_penalized_relevo}\\
\operatorname{div}\boldsymbol{v}_{R} & =  0 &      &\text{in }\Omega\nonumber\\
      \boldsymbol{v}_{R} & = \boldsymbol{0} &      & \text{on }\partial\Omega\nonumber
\end{alignat}

\begin{remark}
Defining the constant $C\geq0$ given by
\[
C=\nu\left\Vert \boldsymbol{g}\right\Vert _{1,\Omega_F}+\kappa\left\Vert
\boldsymbol{g}\right\Vert _{1,\Omega_F}^{2}%
\]
we can proceed similarly than Section IV.2 in \cite{GR86} and conclude that the solutions of  \eqref{NS} and \eqref{NS_relevo} are unique provided
\[
\dfrac{C\kappa}{\left(  \nu-\alpha\right)  ^{2}}<1 .
\]
\end{remark}
Repeating the same arguments as in Section~\ref{stokes}, it is possible to obtain the same convergence results deduced in Theorems~\eqref{stokes1} and \eqref{stokes2}. The first step is the uniformly boundedness of $\{\boldsymbol{v}_{R}\}_{R>0}$

\begin{prop} \label{prop1}
There exists a constant $C>0$ only depending of $\boldsymbol{g}$ such that
\[
\left\vert \boldsymbol{v}_{R}\right\vert _{1,\Omega}\leq\dfrac{C}{\nu-\alpha
}\text{\qquad}R\left\Vert \boldsymbol{v}_{R}\right\Vert _{0,\Omega_S}%
^{2}\leq\dfrac{C^{2}}{\nu-\alpha}%
\]

\end{prop}

\begin{proof}
Let us testing first equation of \eqref{NS_relevo} by $\boldsymbol{w} \in \boldsymbol{V} (\Omega_F)$. Then,
\begin{align*}
\nu\left(  \nabla\boldsymbol{v},\nabla\boldsymbol{w}\right)  _{0,\Omega_F}
+\left(  \left(  \nabla
\boldsymbol{v}\right)  \boldsymbol{v},\boldsymbol{w}\right)  _{0,\Omega_F
} + \left(  \left(  \nabla \boldsymbol{v}\right)  \boldsymbol{g},\boldsymbol{w}\right)  _{0,\Omega_F}
+\left(  \left(  \nabla\boldsymbol{g}\right)  \boldsymbol{v},\boldsymbol{w}\right)  _{0,\Omega_F}
=-\nu\left(  \nabla\boldsymbol{g},\nabla\boldsymbol{w}\right)_{0,\Omega_F}-\left(  \left(  \nabla\boldsymbol{g}\right)  \boldsymbol{g},\boldsymbol{w}\right)  _{0,\Omega_F}%
\end{align*}
Considering the extension to $\Omega$, we have that for all $\boldsymbol{w}\in\boldsymbol{V}\left(\Omega\right) $:
\begin{align*}
\nu\left(  \nabla\boldsymbol{v}%
,\nabla\boldsymbol{w}\right)  _{0,\Omega}+\left(  \left(  \nabla
\boldsymbol{v}\right)  \boldsymbol{v},\boldsymbol{w}\right)  _{0,\Omega
}+\left(  \left(  \nabla\boldsymbol{v}\right)  \boldsymbol{g},\boldsymbol{w}\right)  _{0,\Omega}+\left(  \left(  \nabla\boldsymbol{g}\right)
\boldsymbol{v},\boldsymbol{w}\right)  _{0,\Omega}
=-\nu\left(  \nabla\boldsymbol{g},\nabla\boldsymbol{w}\right)
_{0,\Omega_F}-\left(  \left(  \nabla\boldsymbol{g}\right)  \boldsymbol{g}%
,\boldsymbol{w}\right)  _{0,\Omega_F}%
\end{align*}
Testing the penalized equation \eqref{NS_penalized_relevo} by $\boldsymbol{w} \in \boldsymbol{V} (\Omega)$, we obtain
\begin{align}
&\nu\left(  \nabla\boldsymbol{v}%
_{R},\nabla\boldsymbol{w}\right)  _{0,\Omega}+R\left(  \boldsymbol{v}%
_{R},\boldsymbol{w}\right)  _{0,\Omega_S}+\left(  \left(  \nabla\boldsymbol{v}_R\right)  \boldsymbol{v}_{R}%
,\boldsymbol{w}\right)  _{0,\Omega}+\left(  \left(  \nabla
\boldsymbol{v}_{R}\right)  \boldsymbol{g},\boldsymbol{w}\right)  _{0,\Omega
}+\left(  \left(  \nabla\boldsymbol{g}\right)  \boldsymbol{v}_{R}%
,\boldsymbol{w}\right)  _{0,\Omega}\nonumber\\
=&-\nu\left(  \nabla\boldsymbol{g},\nabla\boldsymbol{w}\right)
_{0,\Omega_F}-\left(  \left(  \nabla\boldsymbol{g}\right)  \boldsymbol{g}%
,\boldsymbol{w}\right)  _{0,\Omega_F}\label{FV_NS_penalized}
\end{align}
Taking $\boldsymbol{w} = \boldsymbol{v}_{R}$, we have
$$
\left(  \left(  \nabla\boldsymbol{v}%
_{R}\right)  \boldsymbol{g},\boldsymbol{v}_R \right)  _{0,\Omega}=\left(  \left(  \nabla\boldsymbol{v}%
_{R}\right)  \boldsymbol{v}_R,\boldsymbol{v}_R \right)  _{0,\Omega}=0
$$
since $\operatorname{div} \boldsymbol{g} = 0$, and then
\begin{align*}
\nu\left\vert \boldsymbol{v}_{R}\right\vert _{1,\Omega}^{2}+\left(  \left(
\nabla\boldsymbol{g}\right)  \boldsymbol{v}_{R},\boldsymbol{v}_{R}\right)
_{0,\Omega}+R\left\Vert \boldsymbol{v}_{R}\right\Vert _{0,\Omega_S}^{2} &
=-\nu\left(  \nabla\boldsymbol{g},\nabla\boldsymbol{v}_R\right)  _{0,\Omega_F%
}-\left(  \left(  \nabla\boldsymbol{g}\right)  \boldsymbol{g},\boldsymbol{v}_R%
\right)  _{0,\Omega_F}
\end{align*}
and due Theorems \ref{Ineq1_NS} and \ref{Ineq2_NS},
\begin{align*}
\left(  \nu-\alpha\right)  \left\vert \boldsymbol{v}_{R}\right\vert
_{1,\Omega}^{2}+R\left\Vert \boldsymbol{v}_{R}\right\Vert _{0,\Omega_S}^{2}
&  \leq\left(  \nu\left\Vert \boldsymbol{g}\right\Vert _{1,\Omega_F}%
+\kappa\left\Vert \boldsymbol{g}\right\Vert _{1,\Omega_F}^{2}\right)
\left\vert \boldsymbol{v}_{R}\right\vert _{1,\Omega}%
\end{align*}
Hence, defining
\[
C=\nu\left\Vert \boldsymbol{g}\right\Vert _{1,\Omega_F}+\kappa\left\Vert
\boldsymbol{g}\right\Vert _{1,\Omega_F}^{2}%
\]
we conclude
\[
\left\vert \boldsymbol{v}_{R}\right\vert _{1,\Omega}\leq\dfrac{C}{\nu-\alpha
}\text{\qquad}R\left\Vert \boldsymbol{v}_{R}\right\Vert _{0,\Omega_S}%
^{2}\leq\dfrac{C^{2}}{\nu-\alpha}%
\]
\end{proof}

The second step is to prove the weakly convergence of $\{\boldsymbol{v}_{R}\}_{R>0}$

\begin{prop}
$\boldsymbol{v}_{R} $ converges weakly to $\boldsymbol{v}$ in $\boldsymbol{V}
(\Omega)$.
\end{prop}

\begin{proof}
From Proposition \ref{prop1}, we see that $\boldsymbol{v}_{R}$ is bounded in $\boldsymbol{H}_{\Gamma_{D}}^{1}\left(  \Omega\right)  =V$ and  $\chi_{\Omega_S}\boldsymbol{v}%
_{R}\rightarrow0$ as $R\rightarrow+\infty$. Then there exists a subsequence of $\boldsymbol{v}_{R}$ (denoted by the same way) that converges weakly in $\boldsymbol{H}^{1}\left(\Omega\right)  $ to a function $\boldsymbol{\tilde{v}} \in \boldsymbol{H}^1 (\Omega)$. In particular,  $\boldsymbol{\tilde{v}}=\boldsymbol{0}$ in $\Omega_S$. Moreover, by Trace Theorem, $\boldsymbol{\tilde{v}}=\boldsymbol{0}$ on $\partial\Omega_S$.
On the other hand, applying \eqref{FV_NS_penalized}, we have that for all $\boldsymbol{w} \in \boldsymbol{V} (\Omega)$:
\begin{align*}
\left(  R\chi_{\Omega_S}\boldsymbol{v}_{R},\boldsymbol{w}\right)
_{0,\Omega}  & =-\left[  \nu\left(  \nabla\left(  \boldsymbol{v}%
_{R}+\boldsymbol{g}\right)  ,\nabla\boldsymbol{w}\right)  _{0,\Omega}%
+\nu\left(  \left(  \nabla\left(  \boldsymbol{v}_{R}+\boldsymbol{g}\right)
\right)  \left(  \boldsymbol{v}_{R}+\boldsymbol{g}\right)  ,\nabla
\boldsymbol{w}\right)  _{0,\Omega}\right]  \\
& \rightarrow-\left[  \nu\left(  \nabla\left(  \boldsymbol{\tilde{v}%
}+\boldsymbol{g}\right)  ,\nabla\boldsymbol{w}\right)  _{0,\Omega}+\nu\left(
\left(  \nabla\left(  \boldsymbol{\tilde{v}}+\boldsymbol{g}\right)  \right)
\left(  \boldsymbol{\tilde{v}}+\boldsymbol{g}\right)  ,\nabla\boldsymbol{w}%
\right)  _{0,\Omega}\right]
\end{align*}
as $R\rightarrow \infty$, since $\boldsymbol{v}_{R}\rightarrow\boldsymbol{\tilde{v}}$ in $\boldsymbol{L}^{p}\left(  \Omega\right)  $, for $p\in\left[  2,6\right)  $.
Then $R\chi_{\Omega_S}\boldsymbol{v}_{R}$ converges weakly to a function
$\boldsymbol{h}\in\left[  \boldsymbol{V}\left(  \Omega\right)  \right]
^{\prime}$ such that $\operatorname{supp}\boldsymbol{h}\subseteq
\Omega_S$.
Then, taking the limit $R \rightarrow \infty$ in \eqref{FV_NS_penalized}, we have that for all $\boldsymbol{w} \in \boldsymbol{V} (\Omega)$:
\begin{equation} \label{FV_limit_Omega}
\nu\left(  \nabla\boldsymbol{\tilde{v}},\nabla\boldsymbol{w}\right)
_{0,\Omega}+\left(  \left(  \nabla\boldsymbol{\tilde{v}}\right)
\boldsymbol{\tilde{v}},\boldsymbol{w}\right)  _{0,\Omega}+\left(
\left(  \nabla\boldsymbol{\tilde{v}}\right)  \boldsymbol{g},\boldsymbol{w}%
\right)  _{0,\Omega}+\left(  \left(  \nabla\boldsymbol{g}\right)
\boldsymbol{\tilde{v}},\boldsymbol{w}\right)  _{0,\Omega}+\left\langle
\boldsymbol{h},\boldsymbol{w}\right\rangle _{V^{\prime},V}=-\nu\left(
\nabla\boldsymbol{g},\nabla\boldsymbol{w}\right)  _{0,\Omega_F}-\left(
\left(  \nabla\boldsymbol{g}\right)  \boldsymbol{g},\boldsymbol{w}\right)
_{0,\Omega_F}%
\end{equation}
Since $\boldsymbol{v}_{R}=0$ on $\partial\Omega$, we have $\boldsymbol{\tilde{v}}=\boldsymbol{0}$ on $\partial\Omega$ due the continuity of the trace operator. By the De Rham's Theorem, there exists $\tilde{p}\in L_{0}^{2}\left(
\Omega\right)  $ such that
\begin{alignat*}{2}
-\nu\triangle\boldsymbol{\tilde{v}}+\left(  \nabla\boldsymbol{\tilde{v}
}\right)  \boldsymbol{\tilde{v}}+\left(  \nabla\boldsymbol{\tilde{v}}\right)
\boldsymbol{g}+\left(\nabla\boldsymbol{g}\right)  \boldsymbol{\tilde{v}}+\nabla\boldsymbol{\tilde{p}}+\boldsymbol{h} &  =  \nu\triangle\boldsymbol{g}
-\left(  \nabla\boldsymbol{g}\right)  \boldsymbol{g} &\quad&\text{in }\Omega\\
\operatorname{div}\boldsymbol{\tilde{v}} &  = 0 & &\text{in }\Omega\\
\boldsymbol{\tilde{v}} &  = \boldsymbol{0} & &\text{on }\partial\Omega
\end{alignat*}
And since $\operatorname{supp} \boldsymbol h \subseteq \Omega_S$, we have that for all $\boldsymbol{w} \in \boldsymbol{V} (\Omega)$ such that $\boldsymbol{w} = \boldsymbol{0}$ on $\Omega_S$,
\begin{align*}
& \nu\left(  \nabla\boldsymbol{\tilde{v}},\nabla\boldsymbol{w}\right)
_{0,\Omega_F}+\left(  \left(  \nabla\boldsymbol{\tilde{v}}\right)
\boldsymbol{\tilde{v}},\boldsymbol{w}\right)  _{0,\Omega_F}+\left(
\left(  \nabla\boldsymbol{\tilde{v}}\right)  \boldsymbol{g},\boldsymbol{w}%
\right)  _{0,\Omega_F}+\left(  \left(  \nabla\boldsymbol{g}\right)
\boldsymbol{\tilde{v}},\boldsymbol{w}\right)  _{0,\Omega_F}
=  & -\nu\left(  \nabla\boldsymbol{g},\nabla\boldsymbol{w}\right)
_{0,\Omega_F}-\left(  \left(  \nabla\boldsymbol{g}\right)  \boldsymbol{g}%
,\boldsymbol{w}\right)  _{0,\Omega_F}%
\end{align*}
so $(\tilde{\boldsymbol{v}}\vert_{\Omega_F}, \tilde{p}\vert_{\Omega_F})$ is a weak solution for \eqref{NS}. Since such a solution is unique, $\left(  \boldsymbol{\tilde{v}},\tilde{p}\right)
=\left(  \boldsymbol{v},p\right)  $ in $\Omega_F$.
Therefore, $\left(  \boldsymbol{\tilde{v}},\tilde{p}\right)  =\left(
v,p\right)  $ in $\Omega$ and $\boldsymbol{v}_{R}\rightharpoonup
\boldsymbol{v}$ in $\boldsymbol{H}_{0}^{1}\left(  \Omega\right)  $.
\end{proof}

\subsection{Main results}
Finally, we can enunciate and prove the strong convergence results.
\begin{thm}
\label{thm_NS_1} Let $R>0$, $\boldsymbol{u}$ be solution of \eqref{NS_relevo} and $\boldsymbol{u}%
_{R}$ solution of \eqref{NS_penalized_relevo}. With the previous assumptions, there is strong
convergence of $\{\boldsymbol{u}_{R}\}_{R>0}$, i.e.,
\[
\lim_{R \rightarrow\infty} \left\vert \boldsymbol{u}_{R} - \boldsymbol{u} \right\vert_{1,\Omega} = 0
\]
and there exists $C>0$ such that for all $R>0$
\[
\left\Vert \boldsymbol{u} - \boldsymbol{u}_{R} \right\Vert_{0, \Omega_S} \leq\dfrac{C}%
{R^{1/2}}
\]

\end{thm}

\begin{proof}
Let $\boldsymbol{w}_{R}=\boldsymbol{v}_{R}-\boldsymbol{v}$. From the variational formulations \eqref{FV_NS_penalized} and \eqref{FV_limit_Omega} we obtain that for all $\boldsymbol{w}\in\boldsymbol{V}\left(
\Omega\right)  $%
\begin{align*}
\nu\left(  \nabla\boldsymbol{w}_{R},\nabla\boldsymbol{w}\right)  _{0,\Omega
}+\left(  \left(  \nabla\boldsymbol{v}_{R}\right)  \boldsymbol{v}_{R}%
,\boldsymbol{w}\right)  _{0,\Omega}-\left(  \left(  \nabla
\boldsymbol{\tilde{v}}\right)  \boldsymbol{\tilde{v}},\boldsymbol{w}%
\right)  _{0,\Omega}+\left(  \left(  \nabla\boldsymbol{w}_{R}\right)
\boldsymbol{g},\boldsymbol{w}\right)  _{0,\Omega_F}+\left(  \left(
\nabla\boldsymbol{g}\right)  \boldsymbol{w}_{R},\boldsymbol{w}\right)
_{0,\Omega_F}  & \\
-\left\langle \boldsymbol{h},\boldsymbol{w}\right\rangle _{V^{\prime}%
,V}+R\left(  \boldsymbol{v}_{R},\boldsymbol{w}\right)  _{0,\Omega_S}  & =0
\end{align*}
Writing in terms of $\boldsymbol{w}_R$,
\begin{align*}
\nu\left(  \nabla\boldsymbol{w}_{R},\nabla\boldsymbol{w}\right)  _{0,\Omega
}+\left(  \left(  \nabla\boldsymbol{w}_{R}\right)  \boldsymbol{v}_{R}%
,\boldsymbol{w}\right)  _{0,\Omega}+\left(  \left(  \nabla
\boldsymbol{\tilde{v}}\right)  \boldsymbol{w}_{R},\boldsymbol{w}\right)
_{0,\Omega}+\left(  \left(  \nabla\boldsymbol{w}_{R}\right)  \boldsymbol{g}%
,\boldsymbol{w}\right)  _{0,\Omega_F}&\\
+\left(  \left(  \nabla\boldsymbol{g}\right)  \boldsymbol{w}_{R}%
,\boldsymbol{w}\right)  _{0,\Omega_F}+R\left(  \boldsymbol{w}_{R}%
,\boldsymbol{w}\right)  _{0,\Omega_S}&=\left\langle \boldsymbol{h}%
,\boldsymbol{w}\right\rangle _{V^{\prime},V}%
\end{align*}
Then we take $\boldsymbol{w}=\boldsymbol{w}_{R}$, and thus
\begin{equation} \label{Energy_wR}
\nu\left\vert \boldsymbol{w}_{R}\right\vert _{1,\Omega}^{2}+\left(  \left(
\nabla\left(  \boldsymbol{\tilde{v}}+\boldsymbol{g}\right)  \right)
\boldsymbol{w}_{R},\boldsymbol{w}_{R}\right)  _{0,\Omega}+R\left\Vert
\boldsymbol{w}_{R}\right\Vert _{0,\Omega_S}^{2}=\left\langle \boldsymbol{h}%
,\boldsymbol{w}_{R}\right\rangle _{V^{\prime},V}%
\end{equation}
Let $c_{2}=\alpha+\dfrac{C\kappa}{\nu-\alpha}$, then
\[
0<c_{2}=\alpha+\dfrac{C\kappa}{\nu-\alpha}<\alpha+\nu-\alpha=\nu
\]
Hence, from Theorems \ref{Ineq1_NS} and \ref{Ineq2_NS},
\begin{align*}
\left(  \left(  \nabla\left(  \boldsymbol{\tilde{v}}+\boldsymbol{g}\right)
\right)  \boldsymbol{w}_{R},\boldsymbol{w}_{R}\right)  _{0,\Omega} &
\leq\left(  \alpha+\kappa\left\vert \boldsymbol{\tilde{v}}\right\vert
_{1,\Omega}\right)  \left\vert \boldsymbol{w}_{R}\right\vert _{1,\Omega}^{2}\\
&  \leq\left(  \alpha+\dfrac{C\kappa}{\nu-\alpha}\right)  \left\vert
\boldsymbol{w}_{R}\right\vert _{1,\Omega}^{2}=c_{2}\left\vert \boldsymbol{w}%
_{R}\right\vert _{1,\Omega}^{2}%
\end{align*}
and using this inequality in \eqref{Energy_wR}, we have
\[
\left(  \nu-c_{2}\right)  \left\vert \boldsymbol{w}_{R}\right\vert _{1,\Omega
}^{2}+R\left\Vert \boldsymbol{w}_{R}\right\Vert _{0,\Omega_S}^{2}%
\leq\left\langle \boldsymbol{h},\boldsymbol{w}_{R}\right\rangle _{V^{\prime
},V}%
\]
Also we have,
\[
\left\langle \boldsymbol{h},\boldsymbol{w}_{R}\right\rangle _{V^{\prime}%
,V}=-\left[  \nu\left(  \nabla\left(  \boldsymbol{\tilde{v}}+\boldsymbol{g}%
\right)  ,\nabla\boldsymbol{w}_{R}\right)  _{0,\Omega}+\nu\left(  \left(
\nabla\left(  \boldsymbol{\tilde{v}}+\boldsymbol{g}\right)  \right)  \left(
\boldsymbol{\tilde{v}}+\boldsymbol{g}\right)  ,\nabla\boldsymbol{w}%
_{R}\right)  _{0,\Omega}\right]  \rightarrow0
\]
as $R \rightarrow \infty$. Therefore
\[
\left(  \nu-c_{2}\right)  \left\vert \boldsymbol{w}_{R}\right\vert _{1,\Omega
}^{2}+R\left\Vert \boldsymbol{w}_{R}\right\Vert _{0,\Omega_S}^{2}%
=\left\langle \boldsymbol{h},\boldsymbol{w}_{R}\right\rangle _{V^{\prime}%
,V}\rightarrow0
\]
so we have proved that  $\left\vert \boldsymbol{w}_{R}\right\vert _{1,\Omega}%
\rightarrow0$ and $\left\Vert \boldsymbol{w}_{R}\right\Vert _{0,\Omega_S%
}=\mathcal{O}\left(  R^{-1/2}\right)  $.
\end{proof}

\begin{thm}
\label{thm_NS_2} Let $R>0$, $\boldsymbol{u}$ be solution of \eqref{NS_relevo} and $\boldsymbol{u}%
_{R}$ solution of \eqref{NS_penalized_relevo}. With the previous assumptions, where we assume in
addition that $\partial\Omega_F$ is piecewise $\mathcal{C}^2$ and $\boldsymbol{u}_{D} \in H^{3/2} (\Omega)$, then there is strong
convergence of $\{u_{R}\}$ in $H^{1}(\Omega)$ and moreover there exists a constant $C>0$ such that for all $R>0$
\[
\vert \boldsymbol{u} - \boldsymbol{u}_{R} \vert_{1, \Omega} \leq\dfrac{C}{R^{1/4}} ,
\qquad \Vert \boldsymbol{u} - \boldsymbol{u}_{R}\Vert_{0, \Omega_S} \leq\dfrac
{C}{R^{3/4}}
\]

\end{thm}

\begin{proof}
We can assume that $\boldsymbol{g}\in\boldsymbol{H}^{2}\left(
\Omega\right)  $ because $\boldsymbol{u}_{D} \in H^{3/2} (\Omega)$, then there is strong
convergence of $\{u_{R}\}$ in $H^{1}(\Omega)$. Reasoning as in Theorem \ref{stokes2}, we can apply a regularity result (see Theorem IX.5.2 in \cite{G11}) and consider $\left(  \boldsymbol{v},p\right)  \in
\boldsymbol{H}^{2}\left(  \Omega\right)  \times\boldsymbol{H}^{1}\left(
\Omega\right)  $. Defining $\boldsymbol{k}\in H^{1/2}\left(  \Omega\right)  $ by
\[
\boldsymbol{k}=-\nu\dfrac{\partial\boldsymbol{v}+\boldsymbol{g}}%
{\partial\boldsymbol{n}}+p\boldsymbol{n}+\dfrac{1}{2}\left(  \left(
\boldsymbol{v}+\boldsymbol{g}\right)  \cdot\boldsymbol{n}\right)  \left(
\boldsymbol{v}+\boldsymbol{g}\right)  = -\nu\dfrac{\partial\boldsymbol{v}+\boldsymbol{g}}%
{\partial\boldsymbol{n}}+p\boldsymbol{n}
\]
and taking $\boldsymbol{w}\in\boldsymbol{V}_{\partial \Omega_S \setminus \Gamma}\left(
\Omega_F \right)  $ extended by $\boldsymbol{0}$ to $\Omega_S$, we have
\begin{align*}
& \nu\left(  \nabla\boldsymbol{v},\nabla\boldsymbol{w}\right)  _{0,\Omega
}+\left(  \left(  \nabla\boldsymbol{v}\right)  \boldsymbol{v},
\boldsymbol{w}\right)  _{0,\Omega}+\left(  \left(  \nabla\boldsymbol{v}%
\right)  \boldsymbol{g},\boldsymbol{w}\right)  _{0,\Omega}+\left(  \left(
\nabla\boldsymbol{g}\right)  \boldsymbol{v},\boldsymbol{w}\right)  _{0,\Omega
}+\left(  \boldsymbol{k},\boldsymbol{w}\right)  _{0,\partial\Omega_S \setminus \Gamma}\\
=  & -\nu\left(  \nabla\boldsymbol{g},\nabla\boldsymbol{w}\right)
_{0,\Omega_F}-\left(  \left(  \nabla\boldsymbol{g}\right)  \boldsymbol{g}%
,\boldsymbol{w}\right)  _{0,\Omega_F}%
\end{align*}
Since
\begin{align*}
& \nu\left(  \nabla\boldsymbol{v},\nabla\boldsymbol{w}\right)  _{0,\Omega
}+\left(  \left(  \nabla\boldsymbol{v}\right)  \boldsymbol{v},
\boldsymbol{w}\right)  _{0,\Omega}+\left(  \left(  \nabla\boldsymbol{v}%
\right)  \boldsymbol{g},\boldsymbol{w}\right)  _{0,\Omega}+\left(  \left(
\nabla\boldsymbol{g}\right)  \boldsymbol{v},\boldsymbol{w}\right)  _{0,\Omega
}+\left\langle \boldsymbol{h},\boldsymbol{w}\right\rangle _{V^{\prime},V}\\
=  & -\nu\left(  \nabla\boldsymbol{g},\nabla\boldsymbol{w}\right)
_{0,\Omega_F}-\left(  \left(  \nabla\boldsymbol{g}\right)  \boldsymbol{g}%
,\boldsymbol{w}\right)  _{0,\Omega_F}%
\end{align*}
we have
\[
\left(  \forall\boldsymbol{w}\in\boldsymbol{V}_{\partial \Omega_S \setminus \Gamma}\left(  \Omega\right)
\right)  \text{\qquad}\left\langle \boldsymbol{h},\boldsymbol{w}\right\rangle
_{V^{\prime},V}=\left(  \boldsymbol{k},\boldsymbol{w}\right)  _{0,\partial\Omega_S \setminus \Gamma}%
\]
Hence, applying Trace Theorem, Hölder inequality and Sobolev Embedding Theorem (see Section 6.6 in \cite{C13}), there exists a constant $C>0$, independent of $R>0$, such that
\begin{align*}
\left(  \nu-c_{2}\right)  \left\vert \boldsymbol{w}_{R}\right\vert _{1,\Omega
}^{2}+R\left\Vert \boldsymbol{w}_{R}\right\Vert _{0,\Omega_S}^{2} &
\leq\left(  \boldsymbol{k},\boldsymbol{w}_{R}\right)  _{0,\partial\Omega_S \setminus \Gamma
}\\
&  \leq C\left\Vert \boldsymbol{k}\right\Vert _{0,\partial\Omega_S%
}\left\vert \boldsymbol{w}_{R}\right\vert _{1,\Omega_S}^{1/2}\left\Vert
\boldsymbol{w}_{R}\right\Vert _{0,\Omega_S}^{1/2}\\
&  \leq\dfrac{\left(  C\left\Vert \boldsymbol{k}\right\Vert _{0,\partial\Omega_S \setminus \Gamma}\right)  ^{2}}{2\left(  \nu R\right)  ^{1/2}}+\dfrac{1}{4}\left(
\nu\left\vert \boldsymbol{w}_{R}\right\vert _{1,\Omega}^{2}+R\left\Vert
\boldsymbol{w}_{R}\right\Vert _{0,\Omega_S}^{2}\right)
\end{align*}
which can be rewritten as
\[
\nu\left\vert \boldsymbol{w}_{R}\right\vert _{1,\Omega}^{2}+R\left\Vert
\boldsymbol{w}_{R}\right\Vert _{0,\Omega_S}^{2}\leq\dfrac{2\left(
C\left\Vert \boldsymbol{k}\right\Vert _{0,\partial\Omega_S \setminus \Gamma}\right)  ^{2}%
}{3\left(  \nu R\right)  ^{1/2}}%
\]
Therefore, $\left\vert \boldsymbol{w}_{R}\right\vert _{1,\Omega
}=\mathcal{O}\left(  R^{-1/4}\right)  $ and $\left\Vert \boldsymbol{w}%
_{R}\right\Vert _{0,\Omega_S}=\mathcal{O}\left(  R^{-3/4}\right)  $, proving this result.
\end{proof}

\section{Numerical examples}\label{num}
In this section, we report a simple 2D numerical experiment to validate the use of fictitious domains in the study of obstacles, and to verify the convergence orders obtained in Sections \ref{stokes} and \ref{ns}. \textcolor{black}{This experiment is motivated for numerical implementations performed in \cite{aguayo2020}, where the approach consists in reconstructing a potential via the minimization of a least-squares functional with a regularization term in order to reconstruct obstacles which could be either immersed or added to the virtual boundary domain.}

First, we consider the domain $\Omega=\Omega_F\cup\overline{\Omega}_S=(-2,2)\times(-1,1)$ given in Figure \ref{fig2} where $\Omega_F\cap\Omega_S=\emptyset$ and $\Omega_S$ is given by
$$
\Omega_S=(-1.1,-0.9)\times(0.4,1)\cup\{(x,y)\in\mathbb{R}^2\mid(x-1)^2+(y-0.5)^2=(0.3)^2\}
$$

\begin{figure}[H]
\centering\includegraphics[height=4.5cm,keepaspectratio]{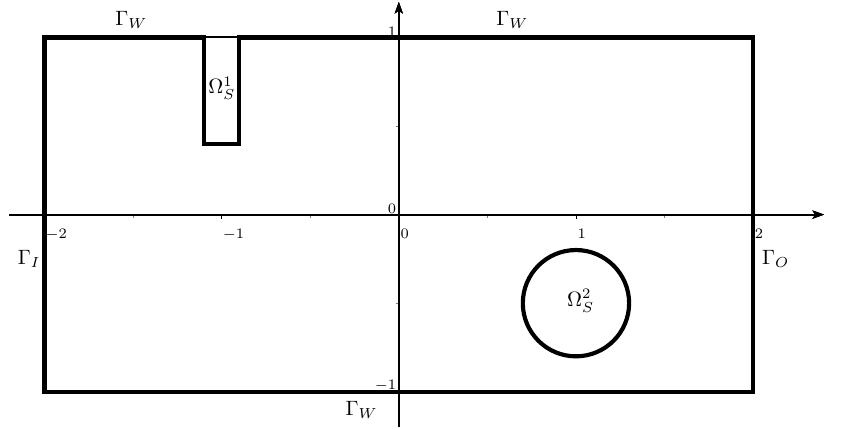}
\caption{Fictitious domain $\Omega$ with obstacles $\Omega_S^1$ and $\Omega_S^2$.}%
\label{fig2}%
\end{figure}

\textcolor{black}{This  example is representative for our purposes, since it consists of a domain with two types of obstacles: the first one, $\Omega_S^1$ is added to the boundary of the whole virtual domain, while the other obstacle, $\Omega_S^2$ is such that its adherence is totally embedded in the fluid.}

In order to determine the reference solutions $(\boldsymbol{u},p)$ for Stokes and Navier-Stokes equations, we consider the following boundary conditions for the domain $\Omega_F$:
\begin{itemize}
\item The inflow $\Gamma_{I}={-2}\times[-1,1]$ has a parabolic profile following Poiseuille's Law given by%
\[
\boldsymbol{u}_{D}\left(  x,y\right)  = - U(1+y)(1-y)  \boldsymbol{n},
\]
where $U>0$, $\boldsymbol{x}=(x,y)$ are the \textcolor{black}{Cartesian} coordinates of the domain and $\boldsymbol{n}$ is the outer normal vector.
\item The do-nothing conditions are imposed on the outflow $\Gamma_{O}={2}\times[-1,1]$, given by%
\[
-\nu\dfrac{\partial\boldsymbol{u}}{\partial\boldsymbol{n}}+p\boldsymbol{n}%
=\boldsymbol{0}.
\]
\item No-slip boundary condition for $\Gamma_{F,W}=\partial\Omega_S\setminus(\Gamma_I\cup\Gamma_O)$.
\end{itemize}
Given $R>0$, we use the same boundary conditions for $\Gamma_{I}={-2}\times[-1,1]$ and $\Gamma_{O}={2}\times[-1,1]$ to calculate the penalized solutions $(\boldsymbol{u}_R,p_R)$. The no-slip boundary condition is now applied to $\Gamma_F=[-2,2]\times\{-1,1\}$.

The numerical solutions of Stokes and Navier-Stokes equations are computed by the Finite Element Method (FEM) with Taylor-Hood elements ($\mathbb{P}_2$ for velocity and $\mathbb{P}_1$ for pressure) on an unstructured triangular mesh generated for $\Omega$ by domain triangulation with $h=0.05$, which corresponds to $8416$ elements and $4329$ nodes. The mesh was designed to approach obstacles $\Omega_S$ as smoothly as possible.

\begin{figure}[H]
\centering\includegraphics[height=4.5cm,keepaspectratio]{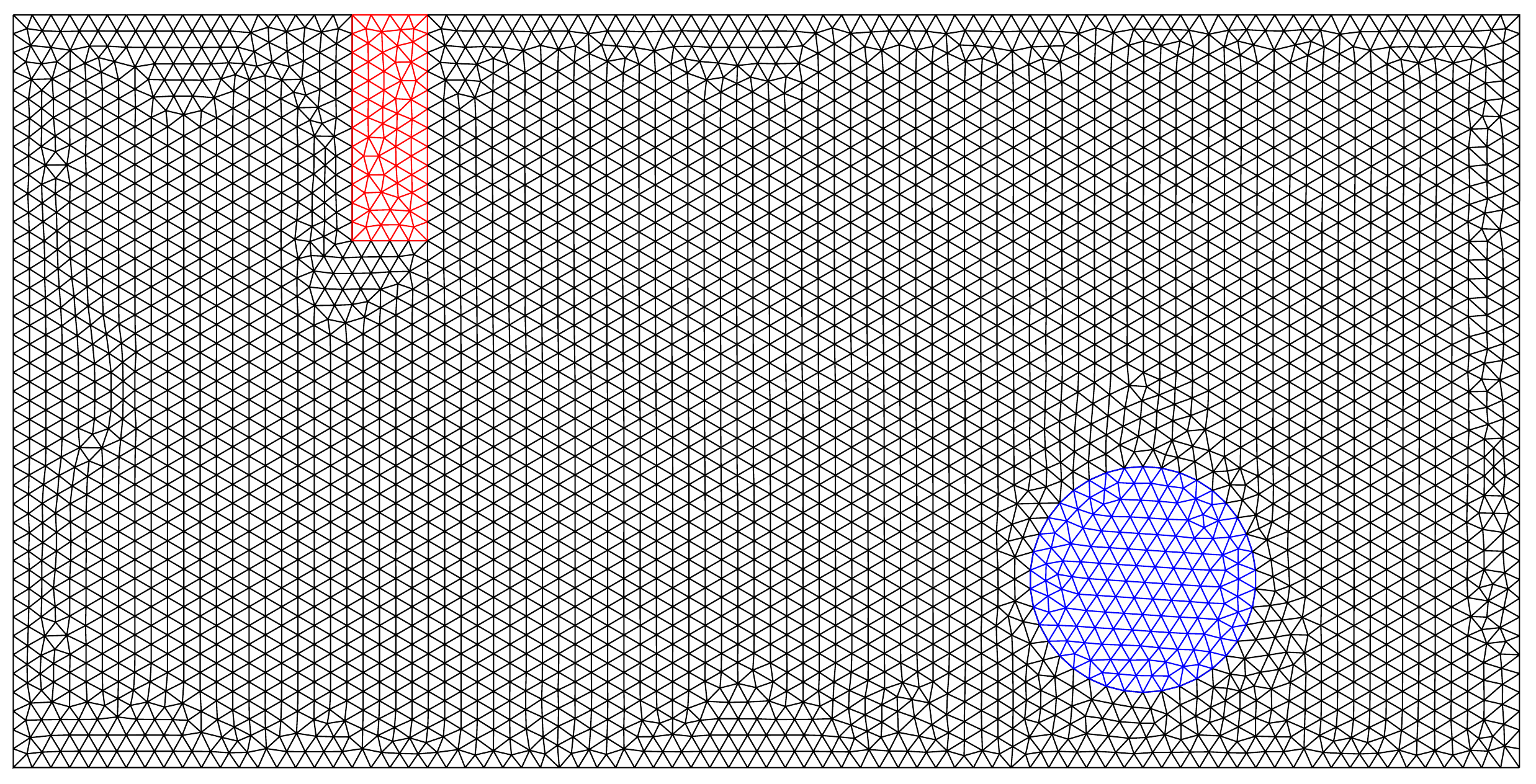}
\caption{Plots of structure mesh of $\Omega_F$ (black), $\Omega_S^1$ (red) and $\Omega_S^2$ (blue).}%
\label{mesh}%
\end{figure}

The mesh is generated by Gmsh \cite{Gmsh} and the numerical solvers are implemented using the Finite element library FEniCS \cite{FENICS} with the default configuration. To solve the nonlinear problems, a Newton's method was used.

The parameters $\nu=1$ and $U=100$ will be used for the Stokes and Navier-Stokes equations. Considering $d=2$ as the length of $\Gamma_I$, the peak Reynolds number on the inflow is
\[
\operatorname{Re}=\dfrac{Ud}{\nu}=200.
\]

\subsection{Stokes equation}
First, we consider the reference solution $(\boldsymbol{u}, p)$ as the solution computed on the real domain $\Omega_F$. Then, for $R\in \{ 10^n\mid n\in\{ 0,1,\ldots,10 \}\}$, the solution $(\boldsymbol{u}_R, p_R)$ is calculated on the fictitious domain $\Omega$. Finally, we compute the errors $\left\Vert \boldsymbol{u}_{R}\right\Vert_{0,\Omega_{S}}$ and $\left\vert \boldsymbol{u}-\boldsymbol{u}_{R}\right\vert _{1,\Omega}$, where $\boldsymbol{u}$ is extended by $\boldsymbol{0}$ on $\Omega_S$.

\begin{table}[H]
\centering{
\begin{tabular}
[c]{|c|cc|cc|}\hline\hline
& \multicolumn{2}{c|}{$\left\Vert \boldsymbol{u}_{R}\right\Vert
_{0,\Omega_{S}}$} & \multicolumn{2}{c|}{$\left\vert \boldsymbol{u}%
-\boldsymbol{u}_{R}\right\vert _{1,\Omega}$}\\
$R$ & Error & Rate & Error & Rate\\\hline\hline
$10^{0}$ & $4.2961\cdot10^{1}$ & $-$ & $4.2961\cdot10^{2}$ & $-$\\
$10^{1}$ & $3.7983\cdot10^{1}$ & $0.0535$ & $3.9663\cdot10^{2}$ & $0.0347$\\
$10^{2}$ & $1.9419\cdot10^{1}$ & $0.2914$ & $2.6400\cdot10^{2}$ & $0.1768$\\
$10^{3}$ & $4.4546\cdot10^{0}$ & $0.6394$ & $1.2827\cdot10^{2}$ & $0.3135$\\
$10^{4}$ & $7.6739\cdot10^{-1}$ & $0.7638$ & $6.2586\cdot10^{1}$ & $0.3116$\\
$10^{5}$ & $1.2696\cdot10^{-1}$ & $0.7813$ & $1.8038\cdot10^{1}$ & $0.5403$\\
$10^{6}$ & $1.5054\cdot10^{-2}$ & $0.9260$ & $2.3898\cdot10^{0}$ & $0.8778$\\
$10^{7}$ & $1.5396\cdot10^{-3}$ & $0.9902$ & $2.4761\cdot10^{-1}$ & $0.9846$\\
$10^{8}$ & $1.5432\cdot10^{-4}$ & $0.9990$ & $2.4851\cdot10^{-2}$ & $0.9984$\\
$10^{9}$ & $1.5436\cdot10^{-5}$ & $0.9999$ & $2.4860\cdot10^{-3}$ & $0.9998$\\
$10^{10}$ & $1.5436\cdot10^{-6}$ & $1.0000$ & $2.4861\cdot10^{-4}$ &
$1.0000$\\\hline\hline
\end{tabular}
}
\caption{History of convergence for Stokes equations.}
\label{TableStokes}
\end{table}%

\begin{figure}[H]
\centering\includegraphics[height=5cm,keepaspectratio]{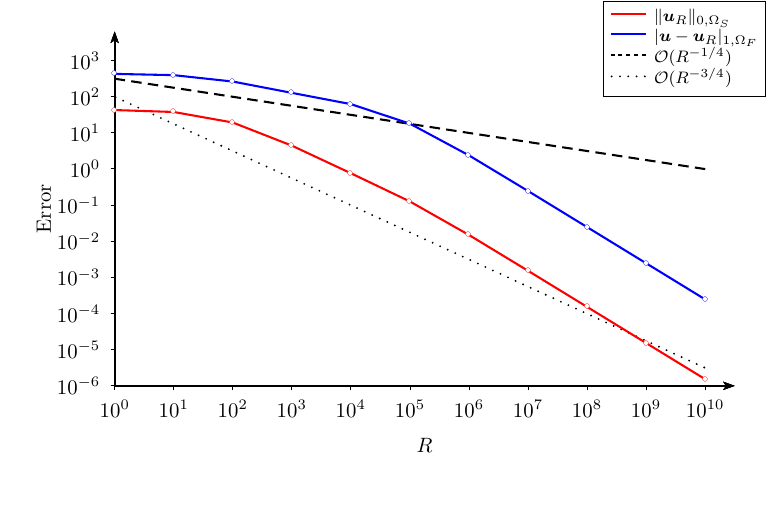}
\caption{History of convergence for Stokes equations.}%
\label{ErrorStokes}%
\end{figure}

Table \eqref{TableStokes} and Figure \ref{ErrorStokes} show that in this experiment the numerical error orders are better than we deduced in Theorem \ref{stokes2}. When $R$ is going to $+\infty$, the errors $\left\Vert \boldsymbol{u}_{R}\right\Vert_{0,\Omega_{S}}$ and $\left\vert \boldsymbol{u}-\boldsymbol{u}_{R}\right\vert _{1,\Omega}$ decrease with order $\mathcal{O}(R^{-1})$. The error orders are similar to $\mathcal{O}(R^{-3/4})$ and $\mathcal{O}(R^{-1/4})$ for $\left\Vert \boldsymbol{u}_{R}\right\Vert_{0,\Omega_{S}}$ and $\left\vert \boldsymbol{u}-\boldsymbol{u}_{R}\right\vert _{1,\Omega}$, respectively, until $R=10^4$. For higher values of $R$, the error order grows up to $\mathcal{O}(R^{-1})$. Hence these results suggest that the obtained error estimates might not be optimal. 

From the isovalues and streamlines plots for $\boldsymbol{u}$ and $\boldsymbol{u}_R$ in Figures \ref{RefStokes}, \ref{Stokes6} and \ref{Stokes2}, we observe that the numerical solution of $\boldsymbol{u}_R$ effectively approximates the reference velocity $\boldsymbol{u}$ for large values of $R$.

\begin{figure}[H]
\captionsetup[subfigure]{labelformat=empty,position=top}
\centering{
\subfloat[Isovalues]{\includegraphics[width=0.40\linewidth,keepaspectratio]{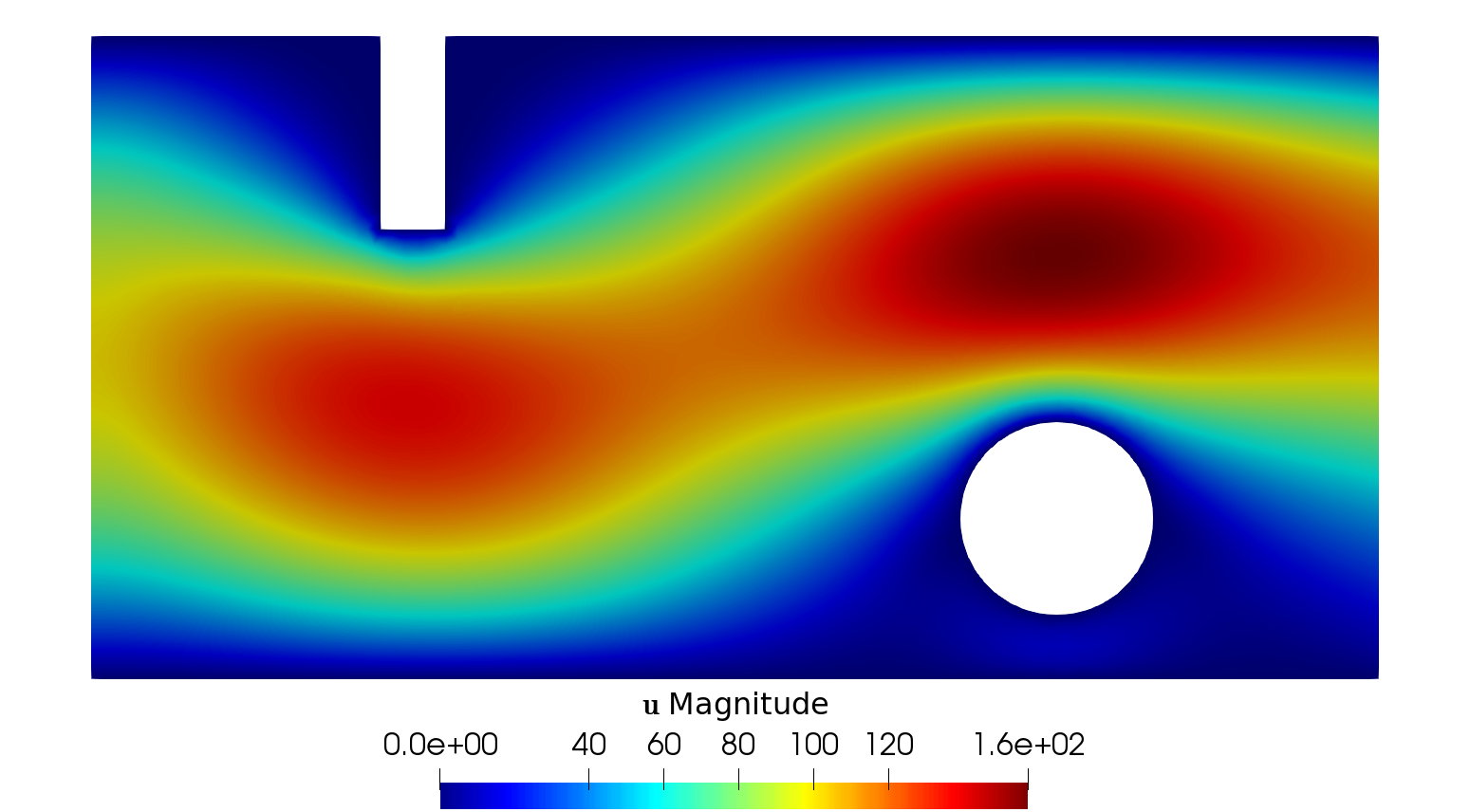}}\quad\quad
\subfloat[Streamlines]{\includegraphics[width=0.40\linewidth,keepaspectratio]{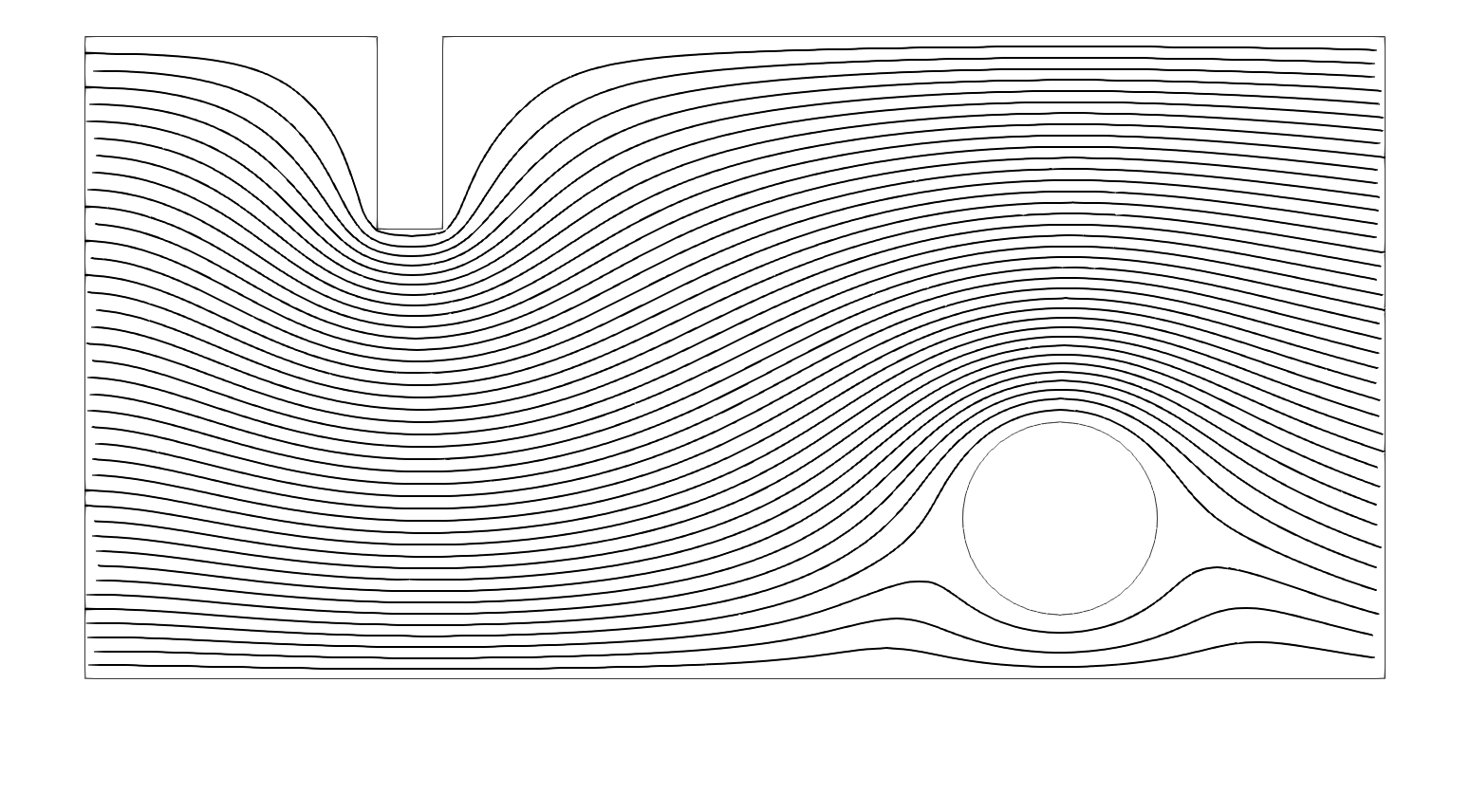}}}
\caption{Reference solution, Stokes equations on $\Omega_F$.}
\label{RefStokes}
\end{figure}

\begin{figure}[H]
\captionsetup[subfigure]{labelformat=empty,position=top}
\centering{
\subfloat[Isovalues]{\includegraphics[width=0.40\linewidth,keepaspectratio]{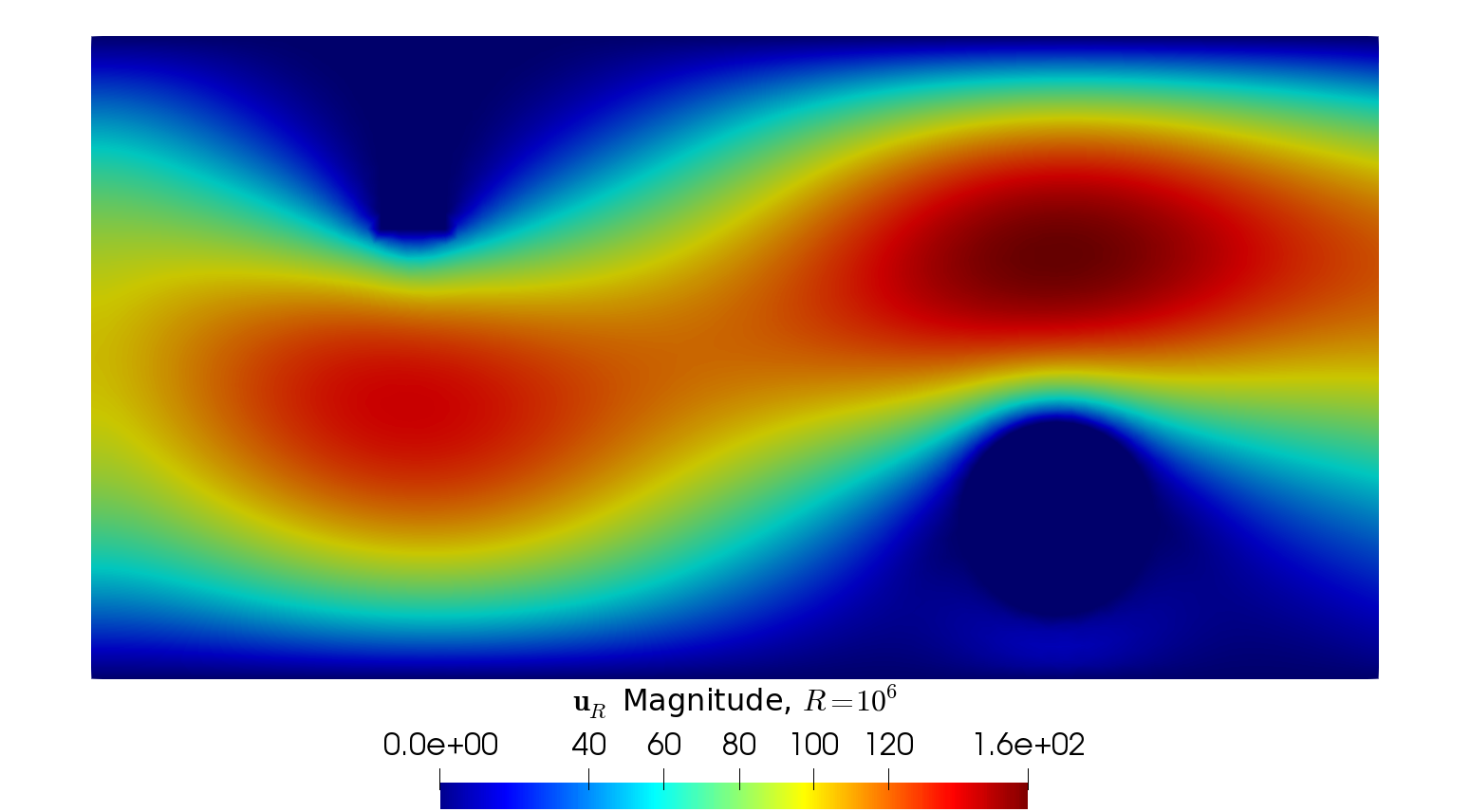}}\quad\quad
\subfloat[Streamlines]{\includegraphics[width=0.40\linewidth,keepaspectratio]{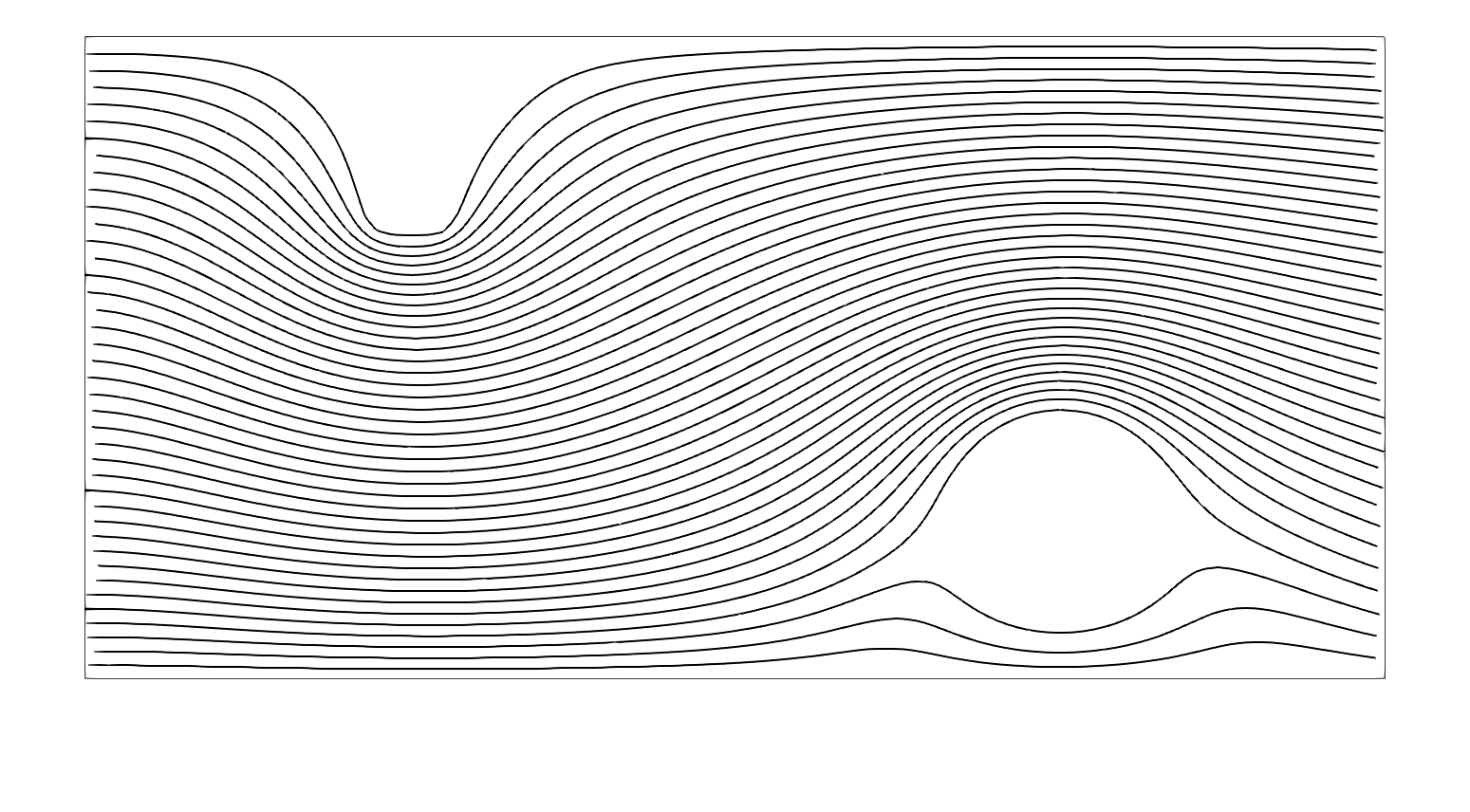}}}
\caption{Solution for penalized Stokes equations for $R=10^6$ on $\Omega$.}
\label{Stokes6}
\end{figure}

\begin{figure}[H]
\captionsetup[subfigure]{labelformat=empty,position=top}
\centering{
\subfloat[Isovalues]{\includegraphics[width=0.40\linewidth,keepaspectratio]{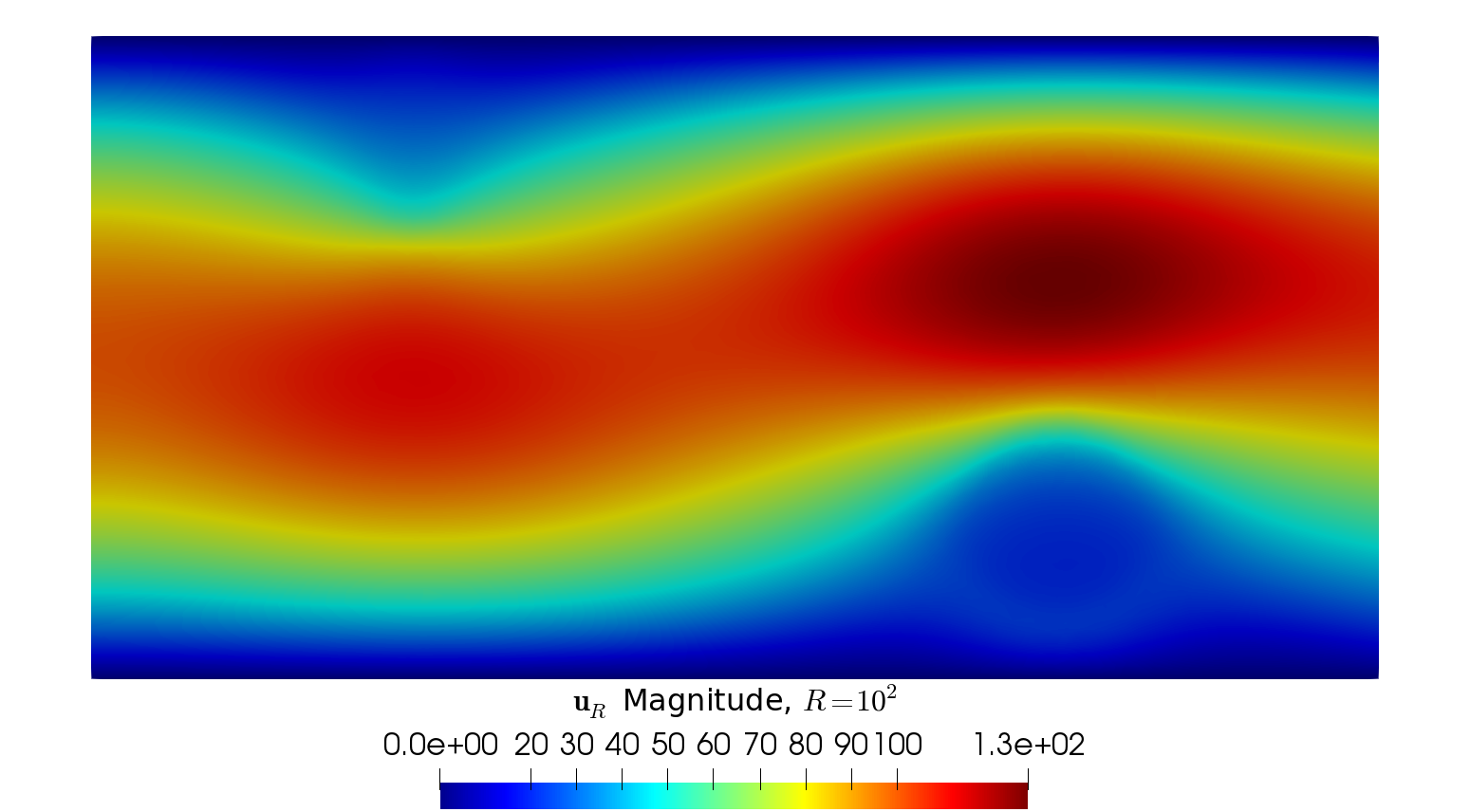}}\quad\quad
\subfloat[Streamlines]{\includegraphics[width=0.40\linewidth,keepaspectratio]{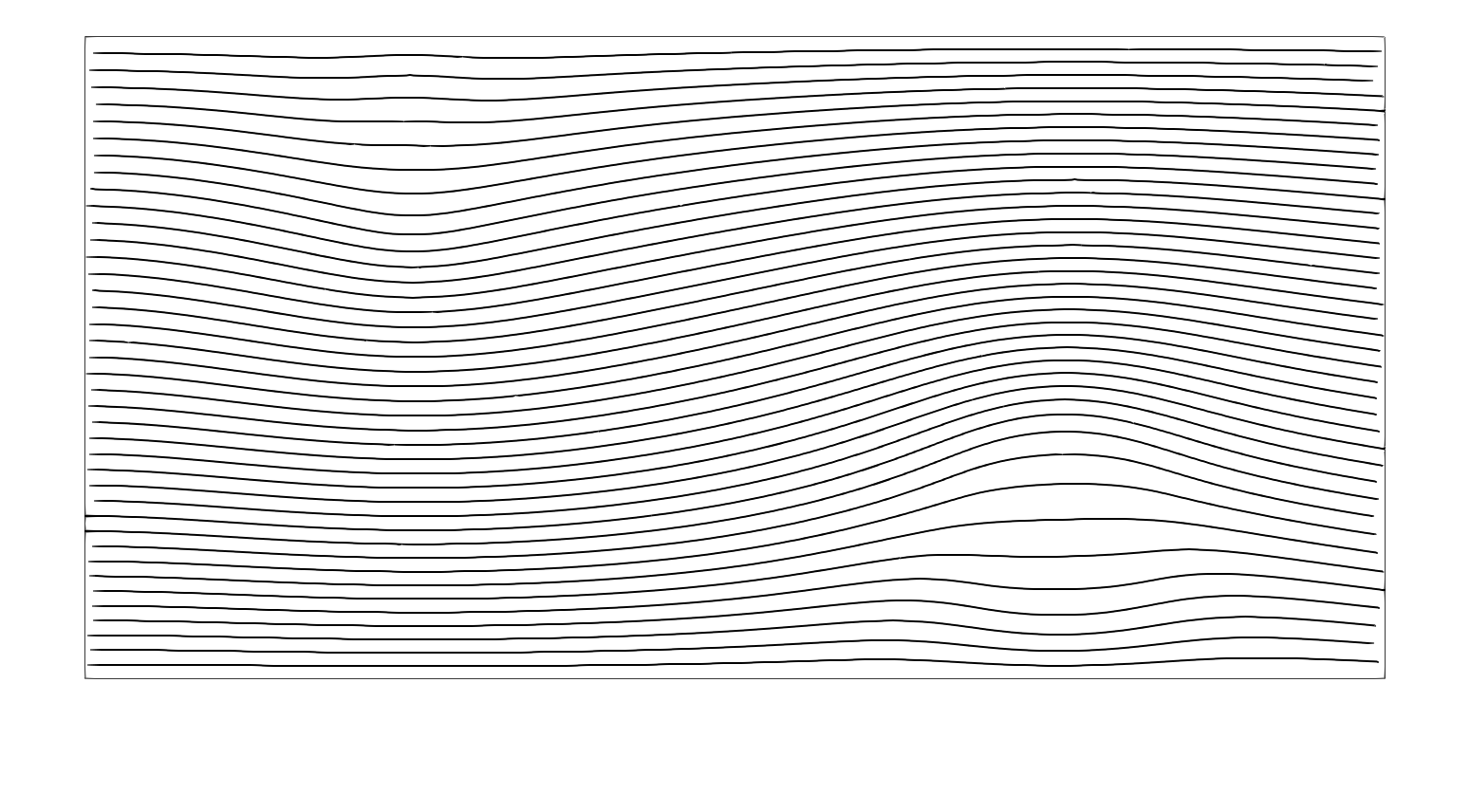}}}
\caption{Solution for penalized Stokes equations for $R=10^2$ on $\Omega$.}
\label{Stokes2}
\end{figure}

\subsection{Navier-Stokes equation}
We repeat the same calculations now for the Navier-Stokes equations. We consider the reference solution $(\boldsymbol{u}, p)$ as the solution computed on the real domain $\Omega_F$ and the solution $(\boldsymbol{u}_R, p_R)$ on the fictitious domain $\Omega$ for $R\in \{ 10^n\mid n\in\{ 0,1,\ldots,10 \}\}$. The errors $\left\Vert \boldsymbol{u}_{R}\right\Vert_{0,\Omega_{S}}$ and $\left\vert \boldsymbol{u}-\boldsymbol{u}_{R}\right\vert _{1,\Omega}$ are computed the same way as in Stokes equations. 

\begin{table}[H]
\centering{
\begin{tabular}
[c]{|c|cc|cc|}\hline\hline
& \multicolumn{2}{c|}{$\left\Vert \boldsymbol{u}_{R}\right\Vert _{0,\Omega
_{S}}$} & \multicolumn{2}{c|}{$\left\vert \boldsymbol{u}-\boldsymbol{u}%
_{R}\right\vert _{1,\Omega}$}\\
$R$ & Error & Rate & Error & Rate\\\hline\hline
$10^{0}$ & $4.3520\cdot10^{1}$ & $-$ & $7.0881\cdot10^{2}$ & $-$\\
$10^{1}$ & $4.2667\cdot10^{1}$ & $0.0086$ & $7.0356\cdot10^{2}$ & $0.0032$\\
$10^{2}$ & $3.5656\cdot10^{1}$ & $0.0780$ & $6.5737\cdot10^{2}$ & $0.0295$\\
$10^{3}$ & $1.5385\cdot10^{1}$ & $0.3650$ & $4.5125\cdot10^{2}$ & $0.1364$\\
$10^{4}$ & $3.0413\cdot10^{0}$ & $0.7040$ & $2.0814\cdot10^{1}$ & $0.3361$\\
$10^{5}$ & $4.6066\cdot10^{-1}$ & $0.8197$ & $6.0166\cdot10^{1}$ & $0.5390$\\
$10^{6}$ & $5.3538\cdot10^{-2}$ & $0.9347$ & $8.0308\cdot10^{0}$ & $0.8746$\\
$10^{7}$ & $5.4639\cdot10^{-3}$ & $0.9912$ & $8.3281\cdot10^{-1}$ & $0.9842$\\
$10^{8}$ & $5.4755\cdot10^{-4}$ & $0.9991$ & $8.3593\cdot10^{-2}$ & $0.9984$\\
$10^{9}$ & $5.4767\cdot10^{-5}$ & $0.9999$ & $8.3625\cdot10^{-3}$ & $0.9998$\\
$10^{10}$ & $5.4768\cdot10^{-6}$ & $1.0000$ & $8.3628\cdot10^{-4}$ &
$1.0000$\\\hline\hline
\end{tabular}
}
\caption{History of convergence for Navier-Stokes equations.}
\label{TableNS}
\end{table}%

\begin{figure}[H]
\centering\includegraphics[height=5cm,keepaspectratio]{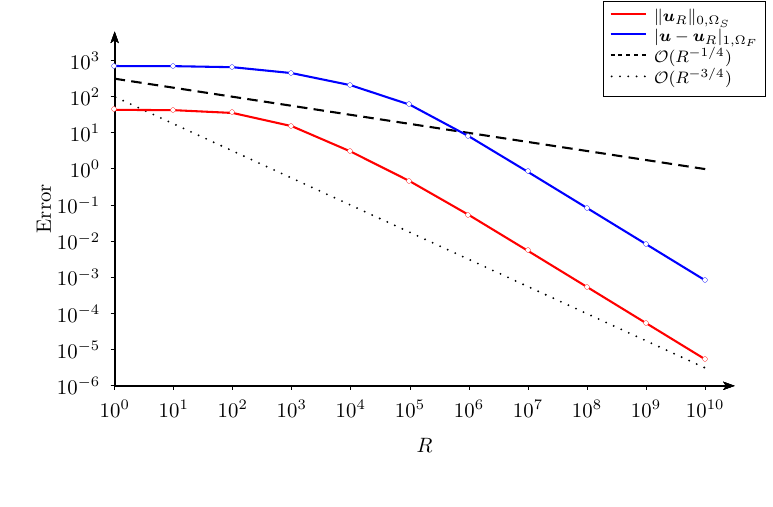}
\caption{History of convergence for Navier-Stokes equations.}%
\label{ErrorNS}%
\end{figure}

While the theory developed in Section \ref{ns} considers only \textcolor{black}{inhomogeneous} Dirichlet boundary conditions, we obtain similar results as in Stokes equations using mixed boundary conditions (Dirichlet and Neumann). Indeed, Table \eqref{TableNS} and Figure \ref{ErrorNS} show that in this experiment the numerical error orders are better than $\mathcal{O}(R^{-3/4})$ and $\mathcal{O}(R^{-1/4})$ for $\left\Vert \boldsymbol{u}_{R}\right\Vert_{0,\Omega_{S}}$ and $\left\vert \boldsymbol{u}-\boldsymbol{u}_{R}\right\vert _{1,\Omega}$ when $R$ is going to $+\infty$, obtaining an order $\mathcal{O}(R^{-1})$ for both errors. Again, the error estimates obtained in Theorem \ref{thm_NS_2} might not be optimal,
\textcolor{black}{with similar conclusions as in the Stokes problem.}

From the isovalues and streamlines plots for $\boldsymbol{u}$ and $\boldsymbol{u}_R$ in Figures \ref{RefStokes}, \ref{Stokes6} and \ref{Stokes2}, we observe that the numerical solution of $\boldsymbol{u}_R$ effectively approximates the reference velocity $\boldsymbol{u}$, including the vortex after the upper obstacle, for large values of $R$.

\begin{figure}[H]
\captionsetup[subfigure]{labelformat=empty,position=top}
\centering{
\subfloat[Isovalues]{\includegraphics[width=0.40\linewidth,keepaspectratio]{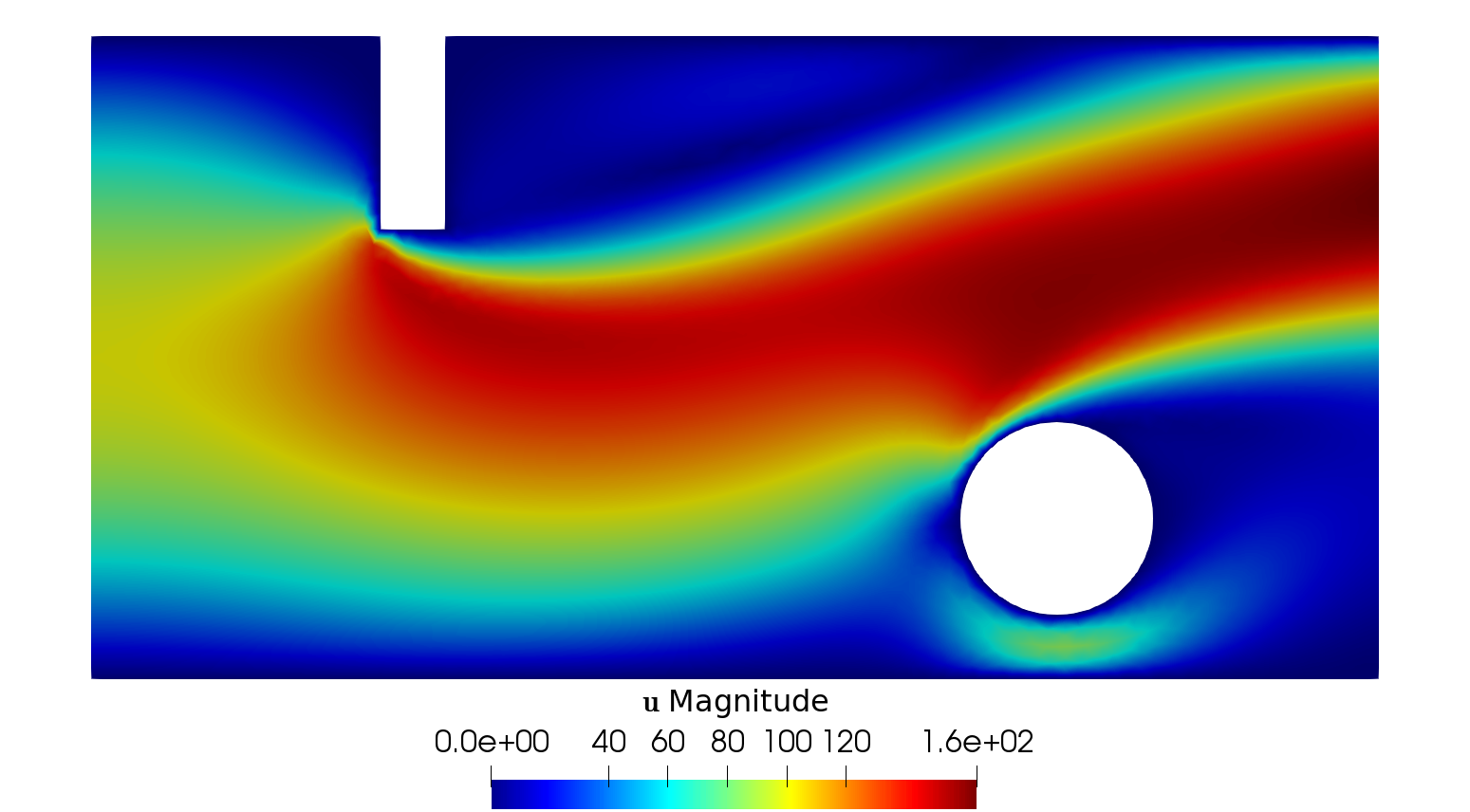}}\quad\quad
\subfloat[Streamlines]{\includegraphics[width=0.40\linewidth,keepaspectratio]{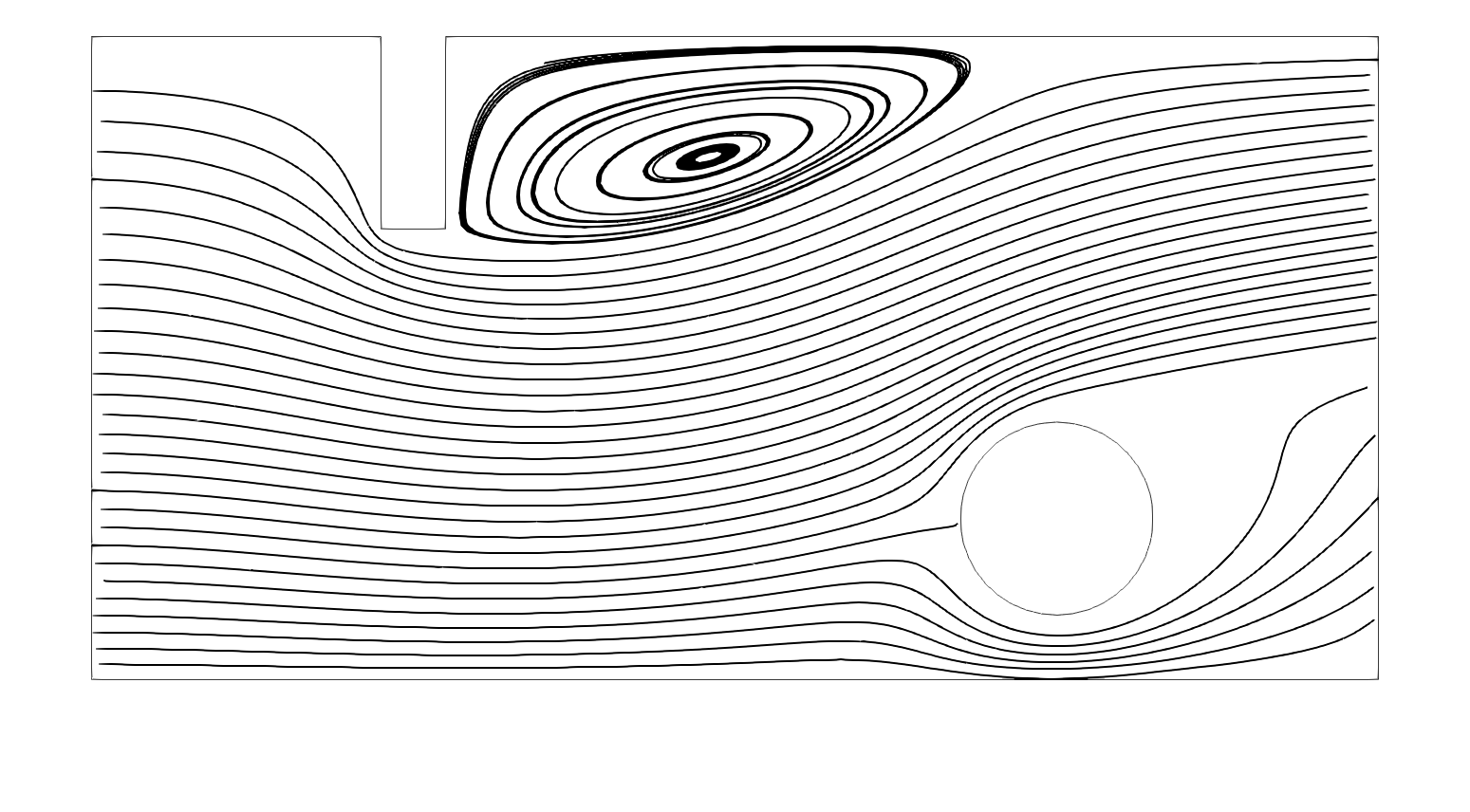}}}
\caption{Reference solution, Navier-Stokes equations on $\Omega_F$.}
\label{RefNS}
\end{figure}

\begin{figure}[H]
\captionsetup[subfigure]{labelformat=empty,position=top}
\centering{
\subfloat[Isovalues]{\includegraphics[width=0.40\linewidth,keepaspectratio]{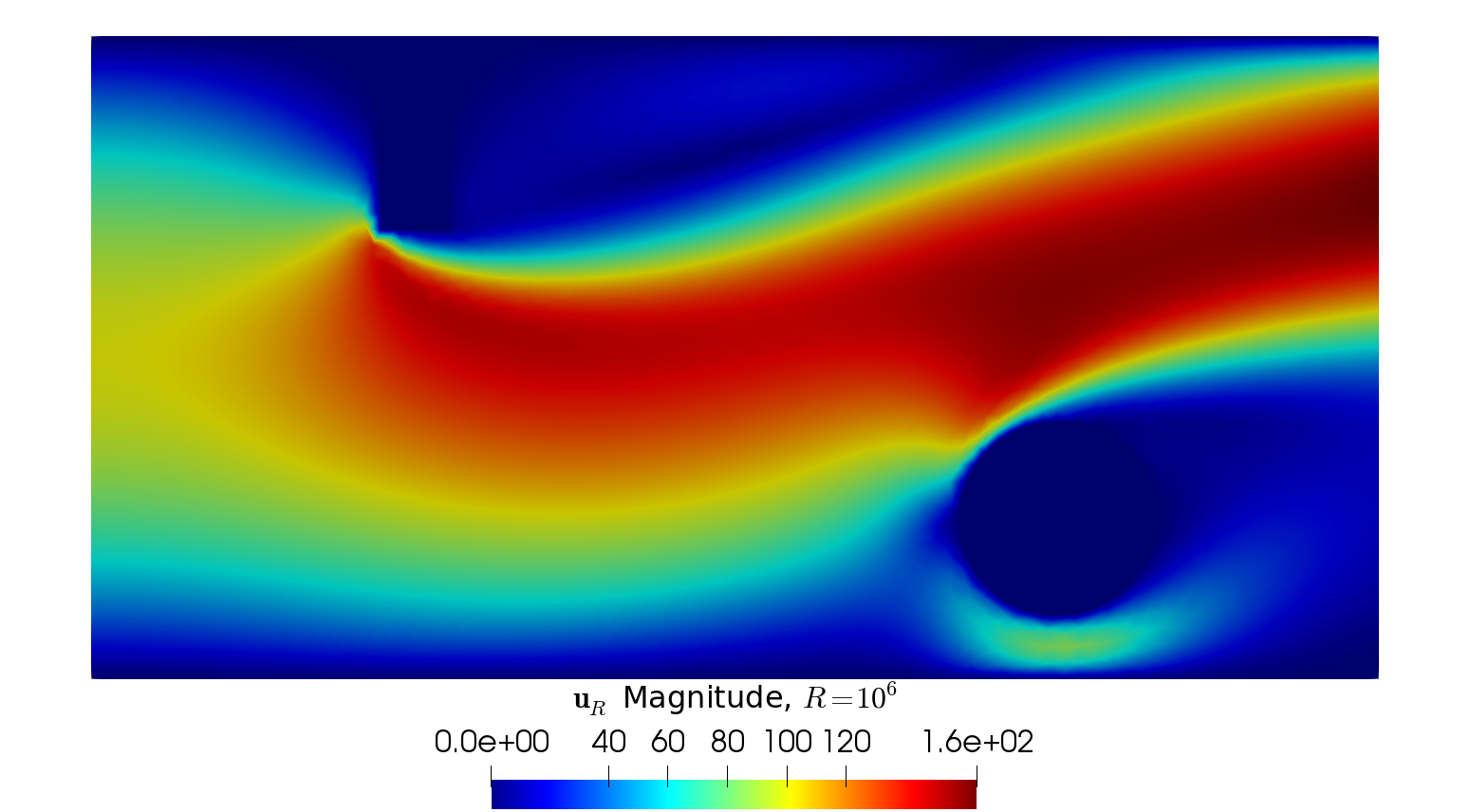}}\quad\quad
\subfloat[Streamlines]{\includegraphics[width=0.40\linewidth,keepaspectratio]{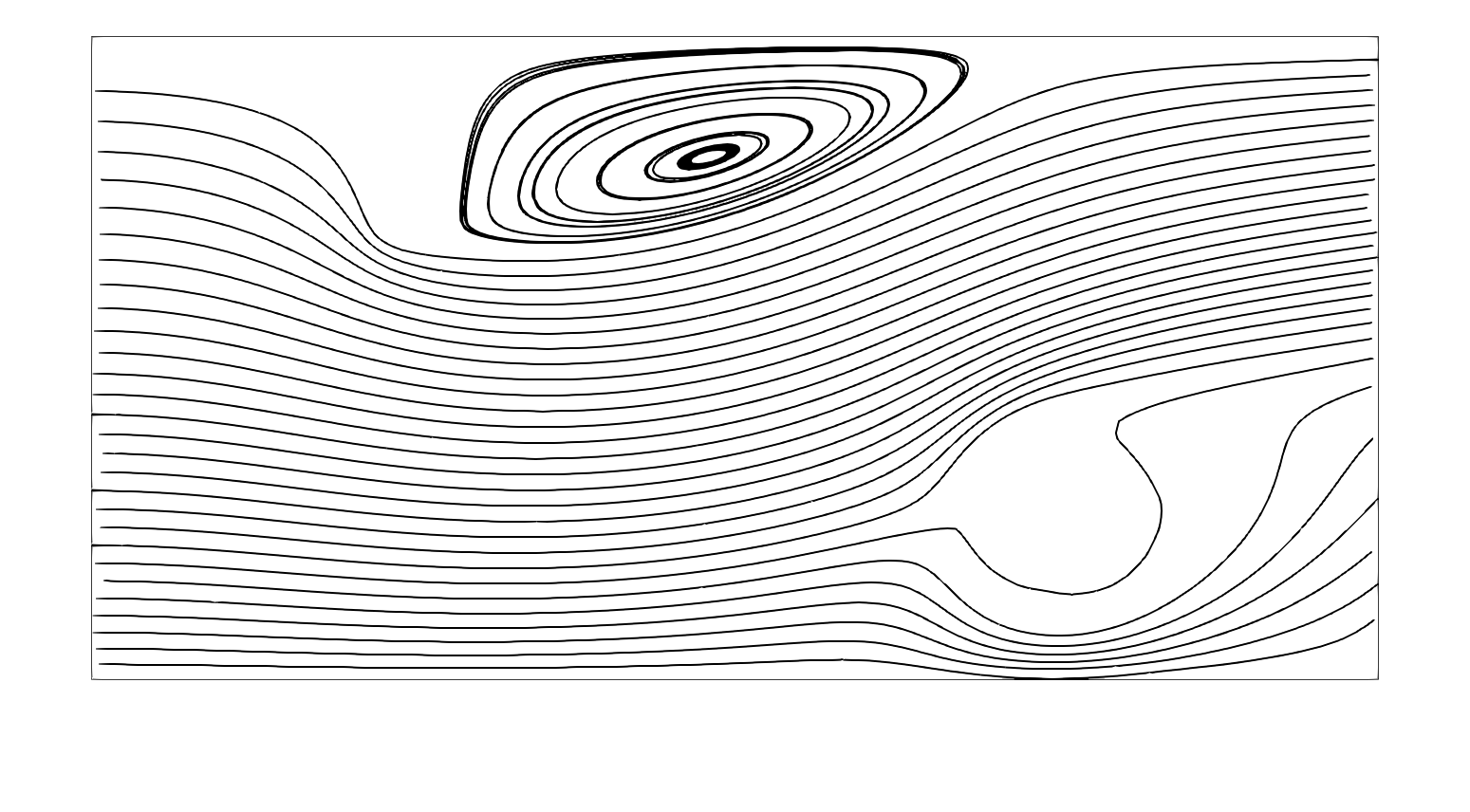}}}
\caption{Solution for penalized Navier-Stokes equations for $R=10^6$ on $\Omega$.}
\label{NS6}
\end{figure}

\begin{figure}[H]
\captionsetup[subfigure]{labelformat=empty,position=top}
\centering{
\subfloat[Isovalues]{\includegraphics[width=0.40\linewidth,keepaspectratio]{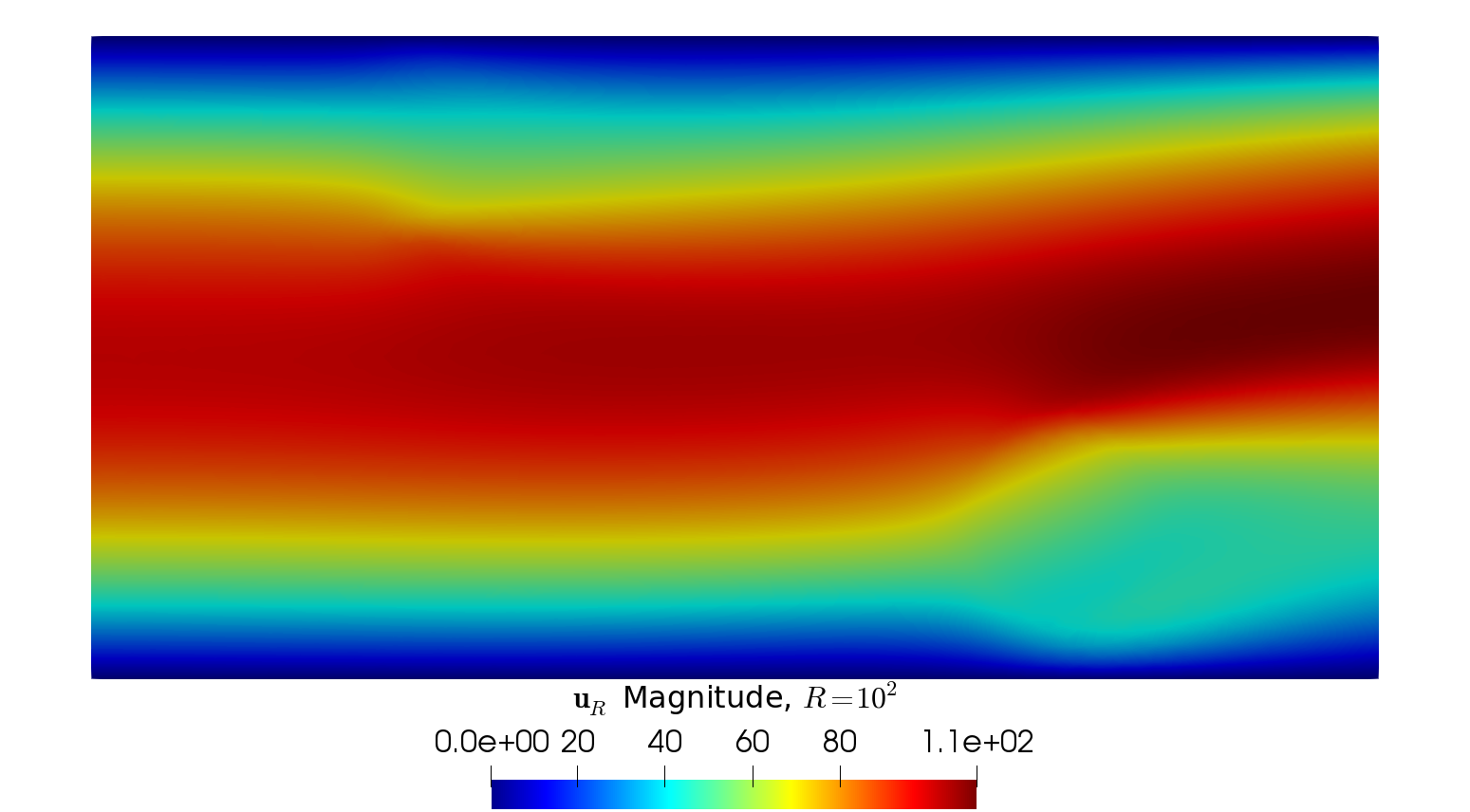}}\quad\quad
\subfloat[Streamlines]{\includegraphics[width=0.40\linewidth,keepaspectratio]{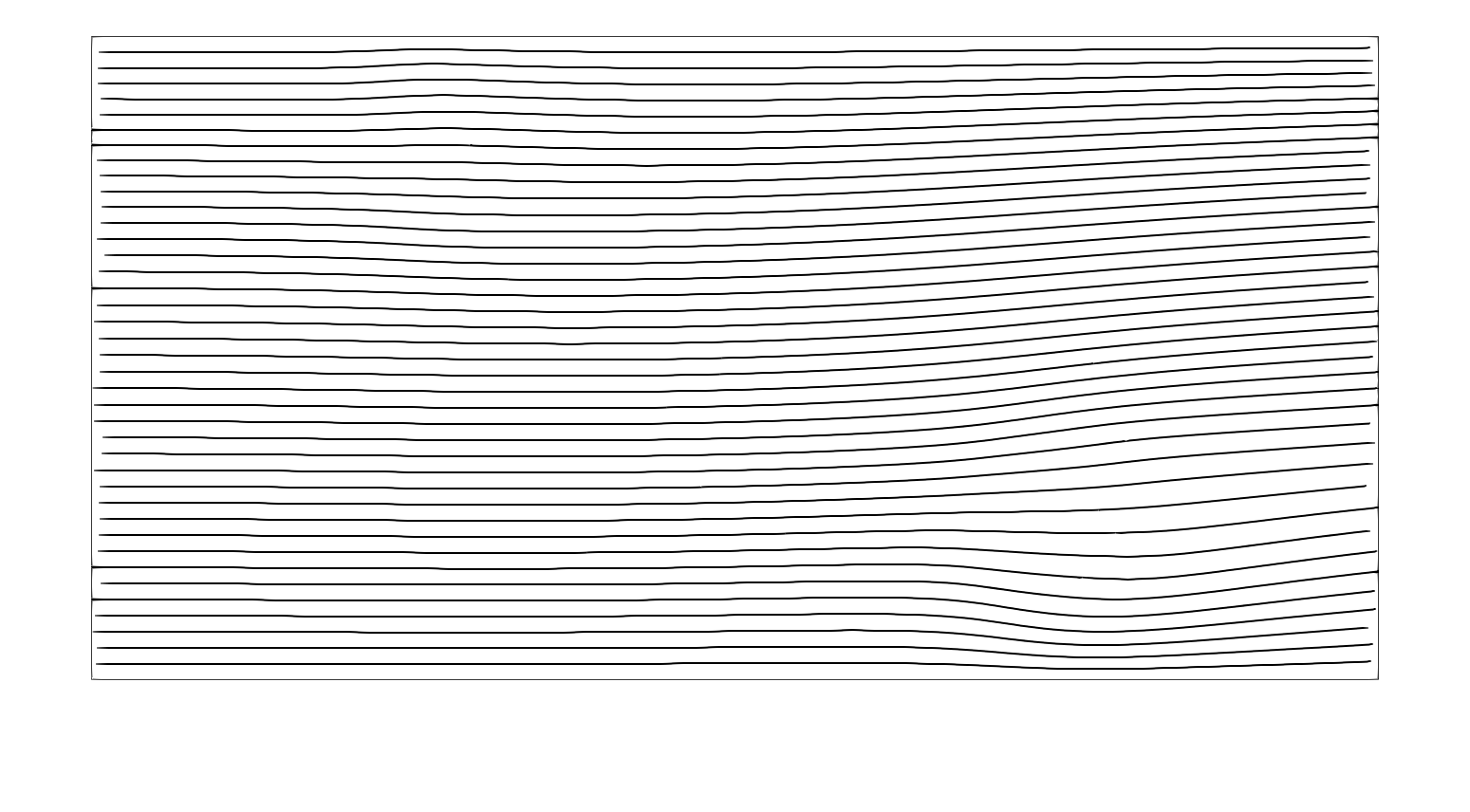}}}
\caption{Solution for penalized Navier-Stokes equations for $R=10^2$ on $\Omega$.}
\label{NS2}
\end{figure}

\section{Conclusions}
We have rigorously established and analyzed a penalization method for steady Stokes and Navier-Stokes equations to approximate the fluid equations around obstacles. The error estimations obtained in Sections \ref{stokes} and \ref{ns} allow us to consider the penalization parameter $R$ as large as necessary to reduce the penalty error, verifying the robustness of the method.

The numerical test proves that this method is easy to implement and is a way to analyze obstacles that does not change the domain when working with a fictitious domain. 
\textcolor{black}{Hence, we have shown both theoretically and with numerical experiments that the equations presented constitute a valid model for fluids going through obstacles in which the numerical implementation is much simpler and cheap for computations since it will not depend on the geometry of the obstacles. As a consequence,} it is possible to avoid shape optimization methods and work with this penalization term \textcolor{black}{and thus to simplify} the models and their numerical implementation.

\textcolor{black}{About future work, the numerical experiments developed in this work allows us to conjecture that it would be possible to improve theoretically the penalty error from $R^{-3/4}$ to $R^{-1}$. In addition, for the time dependent setting, it would be possible to prove similar estimates, and if we use similar techniques, we expect to find the penalty error observed in \cite{ABF99}, since here we obtained the penalty error observed in \cite{A99}.
}

\section*{Acknowledgements}
Jorge Aguayo was partially funded by the National Agency for Research and Development (ANID) / Scholarship Program / BECA DOCTORADO NACIONAL / 2018-21180642. Hugo Carrillo-Lincopi was funded by CMM ANID PIA AFB170001.

\bibliographystyle{elsarticle-harv}
\bibliography{ref}

\end{document}